\documentclass[axioms,article,submit,pdftex,moreauthors]{Definitions/mdpi}
\renewcommand{\linenumbers}{}
\firstpage{1} 
\pubvolume{15}
\issuenum{11}
\articlenumber{853}
\pubyear{2026}
\copyrightyear{2026}
\externaleditor{Hari Mohan Srivastava} 
\datereceived{27 September 2026} 
\daterevised{10 November 2026} 
\dateaccepted{13 November 2026} 
\datepublished{ } 

\usepackage{graphicx}
\usepackage{mathrsfs}

\Title{Asymptotic Formulas for Products of Poisson Random Variables}

\Author{D\v{z}iugas Chvoinikov
 \orcidB{}, Ruslan Prigodin and Jonas \v{S}iaulys *\orcidA{}} 

\AuthorNames{D\v{z}iugas Chvoinikov, Ruslan Prigodin and Jonas \v{S}iaulys}

\address[1]{%
 Institute of Mathematics, Vilnius University, Naugarduko Str. 24, LT-03225 Vilnius, Lithuania; chvoinikov@gmail.com (D.C.); rprigodin@gmail.com (R.P.)}

\corres{Correspondence: jonas.siaulys@mif.vu.lt}

\abstract{In this paper, we consider a collection of independent random variables with Poisson distributions. We first recall the classical properties of sums of independent Poisson random variables. In contrast, the corresponding theory for products of independent Poisson random variables is considerably less developed. This difference motivates our study of products of Poisson random variables. We investigate their distributional properties and derive results that complement the well-established theory for sums. We study the asymptotic behavior of the product tail probability $
{\mathcal{P}_n^{(N)}}= \mathbb{P}(\xi_1\ldots\xi_N \geqslant n),
$
where $\{\xi_1,\ldots,\xi_N\}$ is a finite collection of independent Poisson random variables with positive parameters $\lambda_1,\ldots,\lambda_N$.
We derive a refined Laplace-type asymptotic formula for the tail probability, based on
Stirling's logarithmic approximation, a constrained saddle-point method,
the Lambert function, and a careful evaluation of the constrained
Gaussian prefactor. This yields an approximation formula for $\log {\mathcal{P}_n^{(N)}}$ in the case of large $n$ and fixed $N$. The results obtained for products of Poisson-distributed random variables complement the classical asymptotic theory for sums of such random variables.}


\begin{document}
\nolinenumbers



\section{Introduction}

Random phenomena that occur as discrete events in time or space arise in many areas of science, engineering, and~applied mathematics. Examples include customer arrivals in a queue, radioactive decay events, network packet arrivals, and~the spatial distribution of particles or organisms. A~fundamental probabilistic model used to describe such phenomena is the Poisson distribution, which provides a mathematical framework for modeling the number of events that occur in a fixed interval when they are independent and occur at a constant average~rate.

\begin{itemize}
\item {We say  that a random variable (r.v.) $\xi$  defined in a probability space $(\Omega, \mathcal{F},\mathbb{P})$ is distributed according to the Poisson law if 
$$
\mathbb{P}(\xi=k)=\frac{\lambda^k}{k!}{\rm e}^{-\lambda},\ k\in\{0,1,2\ldots \}\, ,
$$
where $\lambda$ is a positive parameter.} 
\end{itemize}

It is well known that for the r.v. $\xi$ distributed according to the Poisson law with parameter $\lambda$ we have~\cite{h-1967}:
\begin{align*}
  &\mathbb{E}\xi =\lambda,\\ &\mathbb{V}{\rm ar}\,\xi=\mathbb{E}\xi^2-\big(\mathbb{E}\xi\big)^2=\lambda,\\
  &\mathbb{E}\left(\prod_{k=0}^{n-1}(\xi-k)\right)=\lambda^n,\, n\in\mathbb{N},\\
  &\mathbb{E}{\rm e}^{z\xi}={\rm e}^{\lambda(z-1)},\, z\in\mathbb{C}.
\end{align*}

\subsection{Poisson Law as a Limiting~Distribution}

The Poisson distribution plays a central role in probability theory and stochastic processes because it naturally arises as a limit of the binomial distribution when the number of trials becomes large while the probability of success becomes small. 
More precisely, the~following statement holds.
\begin{Theorem}
Let $\xi_n$ be a sequence of independent r.v.s distributed according to the Binomial law, {i.e.,}
$$
\mathbb{P}(\xi_n=k)=\binom{n}{k}p_n^k(1-p_n)^{n-k},\ k\in\{0,1,\ldots, n\}
$$
for some parameters  $p_n\in(0,1)$. If~
$
np_n\mathop{\rightarrow}\limits_{n\rightarrow\infty}\lambda,
$ then 
$$
\mathbb{P}(\xi_n=k)\mathop{\rightarrow}\limits_{n\rightarrow\infty} \frac{\lambda^k}{k!}{\rm e }^{-\lambda}
$$
for any fixed $k\in\{0,1,\ldots\}$.
\end{Theorem}
This classical convergence fact can be proved by using characteristic functions and is usually called \textit{the law of small numbers}. In~fact, this statement  
follows from the so-called {\textit{von Mises theorem}} which has the following~formulation.

\begin{Theorem}\label{du}
Let $\{\xi_{n1},\xi_{n2},\ldots,\xi_{nk_n}\}_{n=1}^\infty$ be a sequence of series of independent Bernoulli  r.v.s such~that 
$$
\mathbb{P}(\xi_{nk}=1)=p_{nk}, \ \mathbb{P}(\xi_{nk}=0)=1-p_{nk},\ p_{nk}\in(0,1),
$$
for $n\in\mathbb{N}$ and $k\in\{1,2,\ldots,k_n\}$ with $k_n\mathop{\rightarrow}\limits_{n\rightarrow\infty}\infty$. Denote
$$
S_n=\sum_{k=1}^{k_n}\xi_{nk}.
$$
Then
$$
\mathbb{P}(S_n=m)\mathop{\rightarrow}\limits_{n\rightarrow\infty}\frac{\lambda^m}{m!}{\rm e}
^{-\lambda}$$
for any fixed $m\in\{0,1,\ldots\}$ if and only if 
\begin{align*}
{\rm (i)} & \max_{1\leqslant k\leqslant k_n}p_{nk}\mathop{\rightarrow}\limits_{n\rightarrow\infty}0,\\
{\rm (ii)}& \sum_{k=1}^{k_n} p_{nk}\mathop{\rightarrow}_{n\rightarrow\infty}\lambda.
\end{align*}
\end{Theorem}
The proofs of the above theorems can be found in~\cite{m-1921, f-1950, l-1960, hl-1960, m-1979}, for~instance.
The presented limiting relationship explains why the Poisson distribution is widely used to approximate rare~events.

It should be noted that condition (i) of Theorem \ref{du} shows the infinitesimality of the system $\{\xi_{n1},\xi_{n2},\ldots,\xi_{nk_n}\}_{n=1}^\infty$. Thus, Theorem \ref{du} shows that the limiting Poisson distribution is very ``strong'', because~it affects the infinitesimality of the converging system. The~following Theorem \ref{trys} shows that in probabilistic number theory, the~Poisson distribution is even ``stronger''. It not only affects the infinitesimality of probabilities but also the difference between the distributions of additive functions and generated random~variables.

\begin{Theorem}[\cite{s-1992,s-1998}]\label{trys}
Let $\big\{f_n(m)\big\}_{n\geqslant 2}$ be a sequence of strongly additive {functions} such that $f_n(q)\in\{0,1\}$ for every prime number $q$, i.e.,~$$  f_n(m)=\sum_{\substack{q|m\\f_n(q)=1}}1.$$
  Then
  $$
  \frac{1}{n}\sum_{\substack{m\leqslant n\\f_n(m)=1}}1\mathop{\rightarrow}\limits_{n\rightarrow\infty}\frac{\lambda^l}{l!}{\rm e}^{-\lambda}
  $$
  for every fixed $l\in\{0,1,\ldots\}$ if and only if
  \begin{align*}
      \lim_{n\rightarrow\infty}\max_{\substack{q\leqslant n\\ f_n(q)=1}}\frac{1}{q}=0,\ \  \lim_{n\rightarrow\infty}\sum_{\substack{q\leqslant n\\f_n(q)=1}}\frac{1}{q}=\lambda,\ \  \lim_{n\rightarrow\infty}\frac{1}{\log n}\sum_{\substack{q\leqslant n\\ f_n(q)=1}}\frac{\log q}{q}=0.
  \end{align*}
\end{Theorem}

\subsection{Poisson~Process}

Additionally, the~Poisson distribution underlies the Poisson process, a~stochastic process that describes the occurrence of events over time or space and is widely used in fields such as telecommunications~\cite{bp-1991, k-1992, dv-2003, aghjw-2010, aaaak-2021}, reliability theory~\cite{b-1996, n-2011, gx-2024, sw-2024, twct-2025, wdo-2025, c-2025}, epidemiology, and~traffic engineering~\cite{gr-2010, apav-2017, fmn-2024, iaur-2024, gjj-2025}.

\begin{itemize}
\item {We  recall that process $N:(\Omega\times[0,\infty))\rightarrow\{0,1,\ldots\}$ is called a homogeneous Poisson process if $\mathbb{P}(N(0)=0)=1$, the~process increments are independent, all process trajectories are continuous from the right, and~$$
\mathbb{P}(N(t-s)=k)={\rm e}^{-\lambda (t-s)}\frac{\big(\lambda(t-s)\big)^k}{k!}
$$
for all $0\leqslant s<t$ and $k\in\{0,1,\ldots\}$, where $\lambda>0$ is the process intensity.}
\end{itemize}

The Poisson distribution and the Poisson process are related but not the same. One is a probability distribution, while the other is a stochastic process. In~fact, the~Poisson process is the simplest process for describing how events occur randomly over time. The~above applications relate to the fact that the Poisson process is the simplest counting renewal process, which counts events in the time interval $[0,t]$.

\subsection{Sum and Product of the Poisson~Distributions}

Let us suppose that r.v.s $\xi_1$ and $\xi_2$ are independent, r.v. $\xi_1$ is distributed according to the Poisson law with parameter $\lambda_1$, and~$\xi_2$ is distributed according to the Poisson law with parameter $\lambda_2$. Since r.v.s $\xi_1$ and $\xi_2$ are independent, we have that 
$$
\mathbb{P}(\xi_1=k,\xi_2=l)={\rm e}^{-(\lambda_1+\lambda_2)}\frac{\lambda^k_1\lambda_2^l}{k!\,l!}
$$
for any $k,l\in\{0,1,\ldots\}$. Therefore, for~$m\in\{0,1,\ldots\}$, we get
\begin{align*}
    \mathbb{P}(\xi_1+\xi_2=m)&=\sum_{k,l\in\mathbb{N}_0:\,k+l=m}{\rm e}^{-(\lambda_1+\lambda_2)}\frac{\lambda^k_1\lambda_2^l}{k!\,l!}\\&=
    \frac{{\rm e}^{-(\lambda_1+\lambda_2)}}{m!}\sum_{k=0}^m\frac{m!}{k!\,(m-k)!}\lambda_1^k\lambda_2^{m-k}\\
    &={\rm e}^{-(\lambda_1+\lambda_2)}\frac{(\lambda_1+\lambda_2)^m}{m!},
\end{align*}
which implies that r.v.s distributed according to the Poisson law are closed with respect to convolution, i.e.,~the sum of two independent r.v.s with Poisson distributions is distributed according to the Poisson~law. 

After expanding the set of Poisson distributions with a shift, we get that the shifted Poisson distributions are closed with respect to the convolution root. This fact follows from the following Raikov's theorem~\cite{r-1937, r-1938}.
\begin{Theorem}
    Let $\xi_1$ and $\xi_2$ be two independent r.v.s and let $\xi_1+\xi_2=\xi$. Suppose that  r.v. $\xi$ is distributed according to the shifted Poisson law $\Pi(\lambda,d)$, i.e.,~
$$
\mathbb{P}(\xi=x)={\rm e}^{-\lambda}\frac{\lambda^{x-d}}{(x-d)!},\, x\in\{d,d+1,d+2,\ldots\}
$$
with $\lambda>0$ and $c\in\mathbb{R}$. 
Then $\xi_1$ is distributed according to  $\Pi(\lambda_1,d_1)$, $\xi_2$ is distributed according to $\Pi(\lambda_2, d_2)$, and~$\lambda_1+\lambda_2=\lambda, d_1+d_2=d$. 
\end{Theorem}

{{In this paper, we consider the product of random variables with Poisson distributions. The~study of products of independent Poisson random variables is interesting for several reasons, both from theoretical and applied perspectives. In~particular, it provides a natural meeting point of probability theory, asymptotic analysis, large deviations, and~the theory of heavy-tailed distributions.}}

\begin{itemize}
\item{An r.v. $\xi$ with distribution $F_\xi(x)=\mathbb{P}(\xi\leqslant x)$ is said to be heavy-tailed, denoted $F_\xi\in\mathscr{H}$, if~$$
\mathbb{E}\,{\rm e}^{\delta \xi}=\int_{-\infty}^\infty {\rm e}^{\delta x}{\rm d}F_\xi(x)=\infty
$$
for any $\delta >0$. Otherwise, r.v. $\xi$ with distribution $F_\xi$ is said to be light-tailed, denoted $F_\xi\in\mathscr{H}^c$.}
\end{itemize}

The following well-known statement shows that the heaviness of a random variable is directly related to the tail $\overline{F}_\xi=1-F_\xi$ of the distribution $F_\xi$, see Theorem 2.6 in~\cite{fkz-2013} or Lemma 1 in~\cite{nwz-2022}.
\begin{Lemma}\label{l1}
    An r.v. $\xi$ with distribution $F_\xi$ is heavy-tailed if and only if
    $$
    \limsup_{x\rightarrow\infty}{\rm e}^{\delta x}\overline{F}_\xi(x)=\infty
    $$
    for any $\delta>0$.
\end{Lemma}

If r.v. $\xi$ is distributed according to the Poisson law $\Pi(\lambda)$ with parameter $\lambda>0$, then
\begin{align*}
    \mathbb{E}{\rm e}^{\delta \xi}={\rm e}^{-\lambda}\sum_{k=1}^\infty\frac{\big(\lambda{\rm e}^\delta\big)^k}{k!}={\rm e}^{\lambda\big({\rm e}^\delta -1\big)}<\infty
\end{align*}
for any $\delta>0$. It follows from this that $F_\xi\in \mathscr{H}^c$ in the~case.

If $\xi_1\mathop{=}\limits^{d}\Pi(\lambda_1)$ and $\xi_2\mathop{=}\limits^{d}\Pi(\lambda_2)$ are two independent r.v.s, then
\begin{align}\label{ar}
    \mathbb{E}{\rm e}^{\delta\xi_1\xi_2}&={\rm e}^{-(\lambda_1+\lambda_2)}\sum_{m=0}^\infty\sum_{k,l\in\mathbb{N}_0:\,kl=m}\frac{{\rm e}^{\,\delta kl}\lambda_1^k\lambda_2^l}{k!\,l!}\nonumber\\
    &={\rm e}^{-(\lambda_1+\lambda_2)}\sum_{k=0}^\infty\sum_{l=0}^\infty \frac{{\rm e}^{\,\delta kl}\lambda_1^k\lambda_2^l}{k!\,l!}\nonumber\\
    &\geqslant{\rm e}^{-(\lambda_1+\lambda_2)}\sum_{k=0}^\infty\frac{{\rm e}^{\,\delta k^2}\big(\lambda_1\lambda_2\big)^k}{\big(k!\big)^2}
\end{align}
for any $\delta>0$. According to Stirling's approximation 
$$
\big(k!\big)^2=2\pi k\bigg(\frac{k}{{\rm e}}\bigg)^{2k}\bigg(1+O\Big(\frac{1}{k}\Big)\bigg).
$$
Therefore, 
\vspace{-2mm}
\begin{align*}
&\quad\quad\quad\quad\qquad\frac{{\rm e}^{\,\delta k^2}\big(\lambda_1\lambda_2\big)^k}{\big(k!\big)^2}\\
&=
\exp\bigg\{\delta k^2+k\log(\lambda_1\lambda_2)-\log 2\pi k-2k\log\frac{k}{{\rm e}}-\log\bigg(1+O\Big(\frac{1}{k}\Big)\bigg)\bigg\}\mathop{\rightarrow}\limits_{k\rightarrow\infty}\infty,
\end{align*}
which implies that the distribution of $\xi_1\xi_2$ is heavy-tailed because 
$$
\mathbb{E}\,{\rm e}^{\delta\xi_1\xi_2}=\infty
$$
for any positive $\delta$ due to estimate \eqref{ar}.

If we multiply even more random variables distributed according to the Poisson law, we get a distribution with an even heavier tail, because~for a collection of independent distributions $$\{\xi_k\mathop{=}\limits^{d}\Pi(\lambda_k), \lambda_k>0\}_{k=1}^N, N\geqslant 3,$$ the following  estimates hold
\begin{align*}
\mathbb{E}\,{\rm e}^{\delta\xi_1\xi_2\ldots\xi_N}&\geqslant\mathbb{E}\,{\rm e}^{\delta\xi_1\xi_2\ldots\xi_N}\mathbb{I}_{\{\xi_N\geqslant 1\}}\\
&\geqslant \big(1-{\rm e}^{-\lambda_N}\big)\mathbb{E}\,{\rm e}^{\delta\xi_1\xi_2\ldots\xi_{N-1}}\\
&\geqslant\big(1-{\rm e}^{-\lambda_{N-1}}\big)\big(1-{\rm e}^{-\lambda_{N}}\big)\mathbb{E}\,{\rm e}^{\delta\xi_1\xi_2\ldots\xi_{N-2}}\\
&\geqslant \ldots \geqslant \prod_{k=3}^N\big(1-{\rm e}^{-\lambda_{k}}\big)\mathbb{E}\,{\rm e}^{\delta\xi_1\xi_2}.
\end{align*}

{{Consequently, the~product of Poisson random variables behaves very differently from their sum. By~summing independent r.v.s that follow the Poisson distribution, we obtain an r.v. that follows the Poisson distribution; only the parameter changes. In~contrast, by~multiplying independent r.v.s distributed according to the Poisson law, we obtain a result that not only falls outside the Poisson distribution family but also fundamentally changes its heaviness level. {
Taking logarithms does not simplify the problem, since
the logarithms of Poisson random variables follow a different
distribution that is also difficult to analyze.
}

In this paper, we will derive the asymptotic formula for the tail of the Poisson product. From~the main result obtained, all the above facts about the heaviness of the Poisson product can be easily obtained using Lemma \ref{l1}. Namely, for~one random variable $\xi$ distributed according to the Poisson law
$$
\mathbb{P}(\xi\geqslant n)\mathop{\sim}\limits_{n\rightarrow\infty}\mathbb{P}(\xi=n)=\frac{{\rm e}^{-\lambda}\lambda^n}{n!}=\exp\big\{-n\log n+O(n)\big\}.
$$
Meanwhile, according to the formula \eqref{eq:log-pnm} for independent Poisson r.v.s $\{\xi_1,\xi_2,\ldots,\xi_N\}$, $N\geqslant 2$,
$$
\mathbb{P}(\xi_1\xi_2\ldots\xi_N\geqslant n)=\exp\Big\{-n^{1/N}\log n+O\big(n^{1/N}\big)\Big\}.
$$
Products of independent Poisson random variables arise naturally when the number of possible configurations or interactions is determined multiplicatively by several independently fluctuating populations. Despite these potential applications, the~probabilistic properties of such products appear to have received considerably less attention than those of sums of Poisson random variables or products of Poisson parameters. However, one can find studies examining models involving products of Poisson-distributed random variables. For~instance, in~\cite{kpz-2016}, such products are used for the construction of confidence intervals for certain real-valued functions of Poisson parameters, whereas in~\cite{gcl-2024, klc-2024, kku-2025, szak-2026}, such products appear in the expression of the likelihood function.}}

\section{Main~Result}

{{In this section, we formulate the main result of the paper.}} In the theorem below, we give an asymptotic formula for the tail of the product of a finite number of independent Poisson random~variables.

\begin{Theorem}\label{t1}
Let \(\{\xi_1,\xi_2,\dots,\xi_N\}\), \(N\geqslant 2\) be a collection of independent Poisson random variables with positive parameters \(\lambda_1,\lambda_2,\dots,\lambda_N\). For~\(n\in\mathbb{N}\), define
\[
\mathcal{P}_n^{(N)}=\mathbb{P}\bigl(\xi_1\xi_2\ldots\xi_N\geqslant n\bigr).
\]
Then, as~\(n\to\infty\),
\begin{align}\label{eq:pnm-main}
\mathcal{P}_n^{(N)}
&=
(2\pi)^{(N-1)/2}
N^{(N-2)/2}
(\log n)^{-(N-1)/2}
n^{(N-1)/(2N)}
\nonumber\\
&\quad\times
\exp\Biggl\{
-\,n^{1/N}\log n
+
n^{1/N}\biggl(
N+\log\prod_{i=1}^N\lambda_i
\biggr)
+
O\left(\frac{n^{1/N}}{\log n}\right)
\Biggr\}.
\end{align}
In particular,
\begin{align}\label{eq:log-pnm}
\log \mathcal{P}_n^{(N)}
&=
-\,n^{1/N}\log n
+
n^{1/N}\biggl(
N+\log\prod_{i=1}^N\lambda_i
\biggr)
+
O\left(\frac{n^{1/N}}{\log n}\right)\nonumber\\
&=\bigg\{-\,n^{1/N}\log n
+
n^{1/N}\biggl(
N+\log\prod_{i=1}^N\lambda_i
\biggr)\bigg\}\left(1+O\Big(\frac{1}{\log^2 n}\Big)\right)\nonumber\\
&=-n^{1/N}\log n
\left(1+{O\left(\frac{1}{\log n}\right)}\right).
\end{align}
\end{Theorem}

\section{Auxiliary~Statements}\label{au}

{{This section and the following Section~\ref{main} are devoted to the proof of our main \mbox{Theorem \ref{t1}}. In~this section, we present several auxiliary lemmas; in Section~\ref{main}, we complete the proof of the theorem.}}
Suppose that r.v.s $\{\xi_1,\xi_2,\ldots,\xi_N\}$ from Theorem \ref{t1} are distributed according to Poisson's laws with parameters $\{\lambda_1,\lambda_2,\ldots,\lambda_N\}$. Our goal is to estimate the probability $\mathcal{P}^{(N)}_n$. {{In the Lemmas  \ref{lem:balanced-regime-poisson} and \ref{lem:first-layer-dominance} below, we select a set of indices in the expression for the probability }}
$$
\mathcal{P}_n^{(N)}=\sum_{k_1k_2\ldots k_N\geqslant n}{\rm e}^{-\sum_{i=1}^N \lambda_i}\prod_{i=1}^N\frac{\lambda_i^{k_i}}{k_i!}\,,
$$
the elements of which have a greater influence on the asymptotics of this probability.
\begin{Lemma}\label{lem:balanced-regime-poisson}
Let $\{\xi_1,\xi_2,\ldots,\xi_N\}$ be a collection of independent Poisson random variables with positive parameters
$\{\lambda_1,\lambda_2,\ldots,\lambda_N\}$, and~let $\mathcal{P}_n^{(N)}$ be the tail probability defined in Theorem \ref{t1}.
For sequence
$
a_n={n^{1/N}}/{\log n}
$
let us denote two index sets: 
\begin{align*}
\mathcal{A}_n
&=
\bigg\{
\mathbf{k}=(k_1,k_2,\dots,k_N):
\prod_{j=1}^N k_j\geqslant n,\ k_i\leqslant \lfloor a_n\rfloor\ \text{for some}\  i\in\{1,2,\ldots,N\}
\bigg\},\\
\mathcal{B}_n
&=
\bigg\{
\mathbf{k}=(k_1,k_2,\dots,k_N):
\prod_{j=1}^N k_j\geqslant n,\  k_i>\lfloor a_n\rfloor\ \text{for all}\ i\in\{1,2,\ldots,N\}
\bigg\}.
\end{align*}
Then  probabilities 
\[
\mathcal{P}_{{A},n}
:=
{\rm e}^{-\sum_{i=1}^N\lambda_i}
\sum_{\mathbf{k}\in \mathcal{A}_n}
\prod_{i=1}^N \frac{\lambda_i^{k_i}}{k_i!} and
\qquad
\mathcal{P}_{B,n}
:=
{\rm e}^{-\sum_{i=1}^N\lambda_i}
\sum_{\mathbf{k}\in \mathcal{B}_n}
\prod_{i=1}^N \frac{\lambda_i^{k_i}}{k_i!}
\]
satisfy the following relationship
\begin{align}\label{ar1}
\frac{\mathcal{P}_{A,n}}{\mathcal{P}_{B,n}} \mathop{\rightarrow}\limits_{n\rightarrow\infty} 0,
\end{align}
implying that
\begin{align}\label{ar2}
\mathcal{P}^{(N)}_n \mathop{\sim}\limits_{n\rightarrow\infty} \mathcal{P}_{B,n}.
\end{align}
\end{Lemma}

\begin{proof}
The identity
\[
\mathcal{P}_n^{(N)}=\mathcal{P}_{A,n}+\mathcal{P}_{B,n}
\]
is immediate, since $\mathcal{A}_n\cap\mathcal{B}_n=\varnothing$ and
\[
\Bigl\{\mathbf{k}:\prod_{i=1}^N k_i\geqslant n\Bigr\}
=
\mathcal{A}_n\cup\mathcal{B}_n.
\]
 Hence, to~prove the relation \eqref{ar2}, it is sufficient to prove the equality \eqref{ar1}. {{We break the rest of the proof down into several steps.}}
 
 \begin{enumerate}[label=\textbf{(i)}, leftmargin=*, labelsep=*]
\item First 
 we estimate the unbalanced contribution $\mathcal{P}_{A,n}$ from above. 
Assume that $\mathbf{k}\in\mathcal{A}_n$. Then for some index $i\in\{1,2,\ldots,N\}$ we have $k_i\leqslant\lfloor a_n\rfloor$ and
 \end{enumerate}
\[
\prod_{j=1}^N k_j \geqslant n.
\]
If, in~addition, all coordinates of $\mathbf{k}$ satisfy 
$$
k_j<b_n:=\left(\frac{n}{a_n}\right)^{1/(N-1)}
= n^{1/N}(\log n)^{1/(N-1)},
$$
then
\[
\prod_{j=1}^N k_j
<
a_n\, b_n^{\,N-1}
=
n,
\]
which is impossible. Hence every $\mathbf{k}\in\mathcal{A}_n$ must satisfy
\[
\max_{1\le j\leqslant N} k_j \geqslant b_n.
\]
Therefore,
\begin{align}\label{ar3}
\mathcal{P}_{A,n}
&=\mathbb{P}\big((\xi_1,\xi_2,\ldots,\xi_N)\in \mathcal{A}_n\big)
\leqslant
\mathbb{P}\Bigl(\max_{1\leqslant j\leqslant N}\xi_j\geqslant b_n\Bigr)\nonumber\\ &=
\mathbb{P}\Bigl(\bigcup_{j=1}^N\{\xi_j\geqslant b_n\}\Bigr)\leqslant
\sum_{j=1}^N \mathbb{P}(\xi_j\geqslant b_n).
\end{align}
{Let us suppose {temporarily} that r.v. $X$ is distributed according to the Poisson law with parameter $\lambda>0$. Due to the Chernoff bound (see, for~instance, pp. 97--98 in~\cite{mu-2005}), we have~that
\begin{align}\label{cher}
\mathbb{P}(X\geqslant u)
\leqslant
\exp\left\{-u\log\frac{u}{\lambda}+u-\lambda\right\}
\end{align}}
if $u>\lambda$. 
Applying this bound with $u=b_n$ {under the condition
$b_n>\max_{1\leqslant j\leqslant N}\lambda_j$},
we obtain
\[
\mathbb{P}(\xi_j\geqslant b_n)
\leqslant
\exp \Big\{-b_n\log\frac{b_n}{\lambda_j}+b_n-\lambda_j\Big\}, j\in\{1,2,\ldots,N\}.
\]
Since
\[
\log b_n
=
\frac1N\log n+\frac{1}{N-1}\log\log n,
\]
it follows that
\[
-b_n\log\frac{b_n}{\lambda_j}+b_n-\lambda_j
=
-\frac1N\,n^{1/N}(\log n)^{N/(N-1)}
+O\Bigl(n^{1/N}(\log n)^{1/(N-1)}\log\log n\Bigr).
\]
for each fixed $j\in\{1,2,\ldots, N\}$. Consequently, the~estimate \eqref{ar3} implies that
\begin{align}\label{ar4}
\mathcal{P}_{A,n}
\leqslant
N\exp\Bigl\{-c_1\,n^{1/N}(\log n)^{N/(N-1)}\Bigr\}
\end{align}
for sufficiently large $n$
{(specifically, such that
$b_n>\max_{1\leqslant j\leqslant N}\lambda_j$)}
and a quantity $c_1$ not depending on $n$.

 \begin{enumerate}[label=\textbf{(ii)}, leftmargin=*, labelsep=*]
\item Next we derive a lower bound for the balanced contribution $\mathcal{P}_{B,n}$.
Set
 \end{enumerate}
\[
r_n:=\big\lceil n^{1/N}\big\rceil.
\]
Then $r_n^N\geqslant n$, and~since
$
a_n=o\big(n^{1/N}\big),
$
we have $r_n>a_n$ for all sufficiently large $n$. Hence
$
(r_n,\dots,r_n)\in\mathcal{B}_n,
$
and, therefore,
\[
\mathcal{P}_{B,n}
\geqslant 
\mathbb{P}\big(\xi_1=r_n,\xi_2=r_n,\ldots,\xi_N=r_n\big)
=
{\rm e}^{-\sum_{i=1}^N\lambda_i}
\prod_{i=1}^N \frac{\lambda_i^{r_n}}{r_n!}.
\]
According to Stirling's formula
\[
\log(r_n!)
=
r_n\log r_n-r_n+O(\log r_n),
\quad n\to\infty.
\]
Consequently,

\begin{align*}
\log \mathcal{P}_{B,n}
&\geqslant
-\sum_{i=1}^N\lambda_i
+
r_n\sum_{i=1}^N \log\lambda_i
-
N\log(r_n!)\nonumber \\
&=
-\sum_{i=1}^N\lambda_i
+
r_n\sum_{i=1}^N \log\lambda_i
-
N r_n\log r_n
+
N r_n
+
O(\log n).
\end{align*}
Since
$
n^{1/N}\leqslant r_n<n^{1/N}+1,
$
and 
$
\log(1+x)\leqslant x,\ x\geqslant 0,
$
we have
\begin{align*}
N r_n\log r_n
=
n^{1/N}\log n+O(\log n).
\end{align*}
Substituting this into the previous estimate for $\log \mathcal{P}_{B,n}$, we obtain
\[
\log \mathcal{P}_{B,n}
\geqslant
-\,n^{1/N}\log n + O(n^{1/N}).
\]
Equivalently, for~ sufficiently large $n$,
\[
\mathcal{P}_{B,n}
\geqslant
\exp\Bigl\{-n^{1/N}\log n-c_2n^{1/N}\Bigr\}
\]
where $c_2$ is a positive quantity not depending on $n$.
Combining the upper bound \eqref{ar4} for $\mathcal{P}_{A,n}$ with the derived lower bound for $\mathcal{P}_{B,n}$, we get
\[
\frac{\mathcal{P}_{A,n}}{\mathcal{P}_{B,n}}
\leqslant
N\exp\!\Bigl\{
-c_1\,n^{1/N}(\log n)^{N/(N-1)}
+n^{1/N}\log n
+c_2n^{1/N}
\Bigr\}
\]
for sufficiently large $n$. The~last estimate implies the relation \eqref{ar1}, which completes the proof of the lemma. 
\end{proof}

Following Lemma \ref{lem:balanced-regime-poisson}, in~the next lemma below we show that inside the balanced region $\mathcal{B}_n$, the~main contribution comes from the points that satisfy $\prod_{i=1}^N k_i\geqslant n$ for the first time, while the contribution from the remaining points in $\mathcal{B}_n$ is asymptotically~negligible.

\begin{Lemma}\label{lem:first-layer-dominance} Let $\{\xi_1,\xi_2,\ldots,\xi_N\}$ be a collection of independent Poisson random variables with positive parameters
$\{\lambda_1,\lambda_2,\ldots,\lambda_N\}$.  Let  $\mathcal{B}_n $ be the same set of indices as in Lemma \ref{lem:balanced-regime-poisson} with the same sequence $a_n$, and~let $\mathcal{P}_{B,n}$ be the probability defined in Lemma   \ref{lem:balanced-regime-poisson}. 
{Define the first balanced layer~by
\begin{align}\label{po}
\mathcal{C}_n
:=
\left\{
\mathbf{k}=(k_1,k_2,\dots,k_N)\in \mathcal{B}_n:
k_1\ldots k_{N-1}(k_N-1)< n
\right\},
\end{align}
and  the corresponding probability by 
\[
\mathcal{P}_{C,n}
:=
{\rm e}^{-\sum_{i=1}^N\lambda_i}
\sum_{\mathbf{k}\in\mathcal{C}_n}
\prod_{i=1}^N \frac{\lambda_i^{k_i}}{k_i!}.
\]
Then
\[
\mathcal{P}_{B,n}
=
\mathcal{P}_{C,n}
\left(
1+O\left(\frac{\log n}{n^{1/N}}\right)
\right).
\]}
\end{Lemma}

\begin{proof}
Fix integers $k_1>\lfloor a_n\rfloor, k_2>\lfloor a_n\rfloor,\dots,k_{N-1}>\lfloor a_n\rfloor$, and~denote {temporarily}
\[
r:=k_1\ldots k_{N-1}.
\]
{
For the proof, we distinguish two~cases.

 \begin{enumerate}[label=\textbf{(i)}, leftmargin=2.15em, labelsep=7mm]
\item First, let us suppose that
 \end{enumerate}
\begin{align}\label{ppp}
\left\lceil\frac{n}{r}\right\rceil
\geqslant
\lfloor a_n\rfloor+1.
\end{align}
}Then the smallest integer \(k_N^*\) such that
\[
(k_1,\ldots,k_{N-1},k_N^*)\in\mathcal B_n
\]
is
\[
k_N^*=\left\lceil\frac{n}{r}\right\rceil.
\]
Moreover,
\[
r(k_N^*-1)<n\leqslant rk_N^*,
\]
and therefore
\[
(k_1,\ldots,k_{N-1},k_N^*)\in\mathcal C_n.
\]
For such fixed choice of $k_1,k_2,\dots,k_{N-1}$, the~contribution to the probability $\mathcal{P}_{B,n}$ from the last coordinate is
\[
{\rm e}^{-\lambda_N}\sum_{\ell\geqslant k_N^*}\frac{\lambda_N^\ell}{\ell!},
\]
while the contribution to $\mathcal{P}_{C,n}$ is just the first term
\[
{\rm e}^{-\lambda_N}\frac{\lambda_N^{k_N^*}}{k_N^*!}.
\]
For any $\ell\geqslant k_N^*$,
\[
\frac{\lambda_N^{\ell+1}}{(\ell +1) !}
\Big/
\frac{\lambda_N^\ell}{\ell !}
=
\frac{\lambda_N}{\ell+1}
\leqslant
\frac{\lambda_N}{k_N^*+1}
\leqslant
\frac{\lambda_N}{\lfloor a_n\rfloor+2}
=
O\!\left(\frac{\log n}{n^{1/N}}\right).
\]
Therefore,
\[
{\rm e}^{-\lambda_N}\sum_{\ell\geqslant k_N^*}
\frac{\lambda_N^\ell}{\ell !}
=
{\rm e}^{-\lambda_N}\frac{\lambda_N^{k_N^*}}{k_N^*!}
\left(
1+O\!\left(\frac{\log n}{n^{1/N}}\right)
\right).
\]
Multiplying this equality by the common factor
\[
{\rm e}^{-\sum_{i=1}^{N-1}\lambda_i}
\prod_{i=1}^{N-1}\frac{\lambda_i^{k_i}}{k_i!}
\]
and summing over all choices {satisfying \eqref{ppp}, we obtain}
\begin{align}\label{1}
\mathcal P_{B,n}^{(1)}
=
\mathcal P_{C,n}
\left(
1+O\!\left(\frac{\log n}{n^{1/N}}\right)
\right),
\end{align}
where $\mathcal P_{B,n}^{(1)}$ {denotes the contribution of the restriction \eqref{ppp} to the probability 
$\mathcal P_{B,n}$.}

{\begin{enumerate}[label=\textbf{(ii)}, leftmargin=*, labelsep=*]
\item   Now let us suppose that
\end{enumerate}
\begin{align}\label{case2}
\left\lceil\frac{n}{r}\right\rceil=\left\lceil\frac{n}{k_1\ldots k_{N-1}}\right\rceil
<
\lfloor a_n\rfloor+1.
\end{align}}
{In this case,
\[
k_N^*=\lfloor a_n\rfloor+1.
\]
Since both sides in the last strict inequality are integers,
\[
\left\lceil\frac{n}{r}\right\rceil
\leqslant
\lfloor a_n\rfloor,
\]
and therefore
\[
\frac{n}{r}\leqslant\lfloor a_n\rfloor.
\]
Hence
\[
r(k_N^*-1)
=
r\lfloor a_n\rfloor
\geqslant n,
\]
and
\[
(k_1,\ldots,k_{N-1},k_N^*)\notin\mathcal C_n,
\]
because membership in $\mathcal C_n$ requires
\[
r(k_N-1)<n.
\]
We now show that the total contribution corresponding to this case is negligible. From~the estimate 
\[
\left\lceil\frac{n}{r}\right\rceil
\leqslant
\lfloor a_n\rfloor
\]
we obtain
\[
r\geqslant\frac{n}{\lfloor a_n\rfloor}
\geqslant\frac{n}{a_n}.
\]
Therefore, at~least one of the coordinates
$k_1,\ldots,k_{N-1}$ satisfies inequality
\[
k_j
\geqslant
\left(\frac{n}{a_n}\right)^{1/(N-1)}
=
n^{1/N}(\log n)^{1/(N-1)}.
\]
Consequently, the~total contribution to term  $\mathcal P_{B,n}$ from all such choices is bounded by
\[
\sum_{j=1}^{N-1}
\mathbb P\left(
\xi_j\geqslant
n^{1/N}(\log n)^{1/(N-1)}
\right).
\]
By the same {Chernoff estimate \eqref{cher} used in the proof} of
Lemma~\ref{lem:balanced-regime-poisson}, this probability  is bounded by
\begin{align}\label{2}
(N-1)
\exp\left\{
-c_3 n^{1/N}(\log n)^{N/(N-1)}
\right\}
\end{align}
for all sufficiently large $n$, where $c_3>0$ does not depend on $n$.\\
On the other hand, choosing
\[
k_1=\ldots=k_{N-1}=\lfloor n^{1/N}\rfloor,
\qquad
k_N=
\left\lceil
\frac{n}{\lfloor n^{1/N}\rfloor^{N-1}}
\right\rceil,
\]
gives a point of $\mathcal C_n$ for all sufficiently large $n$, with~all coordinates equal to
\[
n^{1/N}(1+o(1)).
\]
Hence, by~Stirling's formula,
\begin{align}\label{3}
\mathcal P_{C,n}
\geqslant
\exp\left\{
-n^{1/N}\log n\,(1+o(1))
\right\}.
\end{align}
{Denote the total contribution of the case \eqref{case2} to probability  $\mathcal P_{B,n}$ by
$\mathcal P_{B,n}^{(2)}$.} From the estimates above \eqref{2} and \eqref{3} we have that
\[
\frac{\mathcal P_{B,n}^{(2)}}{\mathcal P_{C,n}}
\leqslant
(N-1)
\exp\left\{
-c_3 n^{1/N}(\log n)^{N/(N-1)}
+n^{1/N}\log n\,(1+o(1))
\right\}
=
o\left(\frac{\log n}{n^{1/N}}\right).
\]

{Finally, by combining the~last  estimate with the estimate \eqref{1}, we obtain}
\[
\mathcal P_{B,n}
=
\mathcal P_{B,n}^{(1)}
+
\mathcal P_{B,n}^{(2)}
=
\mathcal P_{C,n}
\left(
1+O\!\left(\frac{\log n}{n^{1/N}}\right)
\right).
\]
This finishes the proof of the lemma.
}
\end{proof}

\begin{Remark}\label{opaa}
    Due to the assertions of Lemmas 
 \ref{lem:balanced-regime-poisson} and \ref{lem:first-layer-dominance} we have that 
$$
\mathbb{P}\big(\xi_1\xi_2\ldots\xi_N\geqslant n\big)=\mathcal{P}^{(N)}_n\mathop{\sim}\limits_{n\rightarrow\infty}\mathcal{P}_{C,n},
$$
where probability $\mathcal{P}_{C,n}$ is defined in Lemma  \ref{lem:first-layer-dominance}.
\end{Remark}

\section{Proof of the Main Theorem \ref{t1}}\label{main}

In this Section, we prove the asymptotic formula \eqref{eq:pnm-main}. This is sufficient for the proof of Theorem \ref{t1}, because~the equality \eqref{eq:log-pnm} follows from the equality \eqref{eq:pnm-main} immediately.
By \mbox{Lemmas \ref{lem:balanced-regime-poisson} and \ref{lem:first-layer-dominance}}, the~tail probability \(\mathcal{P}_n^{(N)}\) is asymptotically determined by the first balanced layer defined in \eqref{po}, i.e.,
\begin{align}\label{po1}
\mathcal{P}_n^{(N)}\mathop{\sim}_{n\rightarrow\infty}\mathcal{P}_{C,n}&={\rm e}^{-\sum_{i=1}^N\lambda_i}
\sum_{\mathbf{k}\in\mathcal{C}_n}
\prod_{i=1}^N \frac{\lambda_i^{k_i}}{k_i!}\nonumber\\&=\sum_{\mathbf{k}\in\mathcal{C}_n}\mathbb{P}\big(\xi_1=k_1,\xi_2=k_2,\ldots,\xi_N=k_N\big)
\end{align}
where
\[
\begin{aligned}
\mathcal{C}_n
&:=
\left\{ \mathbf{k}=(k_1,k_2,\dots,k_N)\in \mathcal{B}_n: k_1\ldots k_{N-1}(k_N-1)< n \right\},
\\[1mm]
\mathcal{B}_n
&:=
\bigg\{
\mathbf{k}=(k_1,k_2,\dots,k_N):
\prod_{j=1}^N k_j\geqslant n,\  k_i>\lfloor a_n\rfloor \ \text{for all}\ i\in\{1,2,\ldots,N\}
\bigg\}.
\end{aligned}
\]
with $a_n=n^{1/N}/\log n$.

We break the rest of the proof of Theorem \ref{t1} into several~subsections.

\subsection{Setup and System of Lagrange~Equations}

\enlargethispage{1cm}

If $\mathbf{k}=(k_1,k_2,\ldots,k_N)\in\mathcal{C}_n$, then according to the Stirling's approximation formula 
$$
\log k!
=
k\log k-k+\frac12\log(2\pi k)+O\!\left(\frac1k\right),
$$
we have that
\begin{align}\label{eq:T-definition}
&\qquad\qquad\log \mathbb{P}(\xi_1=k_1,\xi_2=k_2,\dots,\xi_N=k_N)\nonumber\\
&= -\sum_{i=1}^N \lambda_i+ \sum_{i=1}^N k_i\bigl(\log\lambda_i-\log k_i+1\bigr)- \frac12\sum_{i=1}^N \log(2\pi k_i)
+ O\!\left(\sum_{i=1}^N \frac1{k_i}\right)\nonumber\\
&\qquad\qquad:=T(\mathbf{k})+O\left(\frac{\log n}{n^{1/N}}\right).
\end{align}

The dominant contribution comes from maximizing expression $T(\mathbf{k})$ under the
constraint
$
\prod\limits_{i=1}^N k_i - n = 0.
$
Let us introduce the Lagrangian
\[
L(\mathbf{k},\alpha)
=
T(\mathbf{k})
- \alpha\Bigl(\prod_{i=1}^N k_i - n\Bigr).
\]
The saddle point  of this Lagrangian $ (\mathbf{k}^*,\alpha^*)=(k_1^*,\dots,k_N^*,\alpha^*)$ satisfies the following equations
\begin{equation}\label{poi}
\begin{cases}
\frac{\partial L}{\partial k_i}\bigg{|}_{\substack{\mathbf{k}=\mathbf{k}^*\\ \alpha=\alpha^*}}
&=
\log\lambda_i-\log k_i^*-\frac{1}{2k^*_i}
-\alpha^*\prod\limits_{\substack{1\leqslant j\leqslant N \\ j\ne i}}k^*_j
=0,
\  i=1,2,\dots,N,\\
\prod\limits_{i=1}^N k^*_i &= n.
\end{cases}
\end{equation}

\subsection{Solution of the Lagrange~System}
\label{sec:lambert-solution}

Let us consider the system \eqref{poi}. According to the second equation of this system 
\[
\prod_{j\ne i} k^*_j = \frac{n}{k^*_i}.
\]
Hence, we can  rewrite the stationarity equations in the form
\[
\log \lambda_i - \log k^*_i - \frac{1}{2k^*_i} - \alpha^* \frac{n}{k^*_i}=0,
\quad i=1,2,\dots,N,
\]
or
\[
k^*_i \exp\left\{\frac{\alpha^* n+\tfrac12}{k^*_i}\right\}=\lambda_i,
\quad i=1,\dots,N.
\]
Now introduce, for~each \(i\),
\[
t_i=-\,\frac{\alpha^* n+\tfrac12}{k^*_i}.
\]
Then 
\[
k^*_i=-\,\frac{\alpha^* n+\tfrac12}{t_i},
\]
and, therefore,
\[
t_i {\rm e}^{t_i}
=
-\,\frac{\alpha^* n+\tfrac12}{\lambda_i}
\]
for any $i\in\{1,2,\ldots,N\}$. This implies that
\[
t_i
=
W\left(
-\frac{\alpha^* n+\tfrac12}{\lambda_i}
\right),
\]
and
\begin{equation}\label{eq:ki-lambert-general}
k_i^*
=
-\,\frac{\alpha^* n+\tfrac12}{
W\!\left(-(\alpha^* n+\tfrac12)/\lambda_i\right)
},
\quad i=1,2,\dots,N,
\end{equation}
where the symbol $W$ denotes the Lambert-$W$ function~\cite{e-1783, cghjk-1996}.

\subsection{Asymptotics of the Critical-Point~Parameters}\label{subsec:s-ki}

Let us consider expressions \eqref{eq:ki-lambert-general}. If~$\alpha^*\geqslant 0$, then \(k^*_i<\lambda_i\) for all \(i\), and~therefore
\[
n=\prod_{i=1}^N k^*_i<\prod_{i=1}^N \lambda_i,
\]
which is impossible for sufficiently large \(n\). Hence, in~the asymptotic regime
\(n\to\infty\) and necessarily \(\alpha^*<0\). Let us denote temporarily 
\[
s:=-\alpha^* n>0.
\]

In this subsection, we analyze  the quantities
\begin{equation}\label{pai}
k_i^*
=
\frac{s-\tfrac12}{
W\!\bigl((s-\tfrac12)/\lambda_i\bigr)},
\quad i=1,\dots,N,
\quad
s=-\alpha^* n>0,
\end{equation}
obtained from the Lagrange Equation \eqref{poi}.

From the Lambert-$W$ representation \eqref{pai} and the constraint
$
\prod_{i=1}^N k_i^* = n,
$
we get
\begin{equation}\label{eq:s-implicit-general}
(s-\tfrac12)^N
=
n\prod_{i=1}^N
W\!\left(\frac{s-\tfrac12}{\lambda_i}\right).
\end{equation}
Since \(s\mathop{\rightarrow}\limits_{n\rightarrow\infty}\infty\), {we may use the large-argument expansion from~\cite{cghjk-1996}:}
\[
W(x)=\log x-\log\log x+o(1),
\quad x\to\infty.
\]
Applied to \(x=(s-\tfrac12)/\lambda_i\), this gives
\begin{equation}\label{ka}
W\left(\frac{s-\tfrac12}{\lambda_i}\right)
=
\log s-\log\log s+{O(1)}
\end{equation}
for any $i\in\{1,2,\ldots,N\}$ as $s\to\infty$.
Substituting this into \eqref{eq:s-implicit-general} and dividing by \((\log s)^N\), we obtain 
\[
\frac{\big(s-\frac12\big)^N}{(\log s)^N}
=
n\left(
1-\frac{\log\log s}{\log s}
+{O\!\left(\frac{1}{\log s}\right)}
\right)^N.
\]
Taking \(N\)-th roots gives
\begin{equation}\label{eq:s-equation-general}
\frac{s}{\log s}
=
n^{1/N}(1+o(1))
\end{equation}
as $s\to \infty$ or equivalently as $n\to\infty$. 
Taking logarithms gives
\[
\log s-\log\log s
=
\frac1N\log n+\log(1+o(1))
=
\frac{\log n}{N}+o(1),
\]
and, consequently,
\[
\log s
=
\frac{\log n}{N}
+o(1)
+\frac{\log\log s}{\log n}\log n
=
\frac{\log n}{N}(1+o(1)).
\]
Substituting this back into \eqref{eq:s-equation-general}
we obtain
\begin{equation}\label{eq:s-final-general}
s
=
n^{1/N}\log s(1+o(1))
=
n^{1/N}\frac{\log n}{N}(1+o(1)),
\end{equation}
or, equivalently,
\begin{equation}\label{eq:alpha-final-general}
\alpha^*
=
-\frac{s}{n}
=
-\frac{\log n}{N}\,n^{-(N-1)/N}(1+o(1)).
\end{equation}
Finally, using the Lambert-$W$ representation \eqref{ka} and relations \eqref{eq:s-equation-general}, \eqref{eq:s-final-general}, we get that
\begin{equation}\label{eq:kstar-final-general}
k_i^*
=
\frac{s-\tfrac12}{W((s-\tfrac12)/\lambda_i)}
=
\frac{s(1+o(1))}
{\log({s}/{\log s})+{O(1)}}
=
\frac{s}{\log s}(1+o(1))
=
n^{1/N}(1+o(1))
\end{equation}
for any $ i\in\{1,2,\dots,N\}$.

\subsection{Reduction to $N-1$ Dimensions on the First Balanced~Layer}\label{fa}

In this subsection, we reduce the first-layer contribution to \(N-1\) free variables.
Recall~that
\begin{align*}
&\qquad\mathcal{C}_n
=
\big\{
\mathbf{k}=(k_1,k_2,\dots,k_N)\in \mathcal{B}_n:
k_1\ldots k_{N-1}(k_N-1)< n
\big\}\\
&=\Big\{\mathbf{k}=(k_1,\dots,k_N):\prod_{i=1}^Nk_i\geqslant n, (k_N-1)\prod_{i=1}^{N-1}k_i<n, k_i>\lfloor a_n\rfloor \ \text{for}\ i\in\{1,\ldots,N\}\Big\}
\end{align*}
with $a_n=n^{1/N}/\log n$.
If  $\mathbf{k}=(k_1,k_2,\dots,k_{N})\in\mathcal{C}_n$, then
$$
k_{N}-1<\frac{n}{k_1\ldots k_{N-1}}\leqslant k_N
$$
implying that
\begin{equation}\label{afi}
k_N=\left\lceil\frac{n}{k_1\ldots k_{N-1}}\right\rceil.
\end{equation}
In the opposite case, if~$k_N$ has  expression \eqref{afi}, then  
$$
\frac{n}{k_1\ldots k_{N-1}}+1>k_N=\left\lceil\frac{n}{k_1\ldots k_{N-1}}\right\rceil\geqslant\frac{n}{k_1\ldots k_{N-1}}
$$
implying that 
$$
\prod_{i=1}^Nk_i\geqslant n\quad \text{and}\quad (k_N-1)\prod_{i=1}^{N-1}k_i<n 
$$
Due to the above considerations, we get that
\vspace{-10pt}
\begin{adjustwidth}{-\extralength}{0cm}
\centering 
$$
\mathcal{C}_n=\left\{\mathbf{k}=\Big(k_1,\ldots,k_{N-1},\Big\lceil\frac{n}{k_1\ldots k_{N-1}}\Big\rceil\Big):k_i>\lfloor a_n\rfloor, 1\leqslant i\leqslant N-1, \Big\lceil\frac{n}{k_1\ldots k_{N-1}}\Big\rceil>\lfloor a_n\rfloor\right\}.
$$
\end{adjustwidth}
We also observe that
$$
\lceil x\rceil>\lfloor a_n\rfloor\ \iff \ x>\lfloor a_n\rfloor.
$$
Hence
$$
\mathcal{C}_n=\left\{\mathbf{k}=\Big(k_1,\ldots,k_{N-1},\Big\lceil\frac{n}{k_1\ldots k_{N-1}}\Big\rceil\Big):k_i>\lfloor a_n\rfloor, 1\leqslant i\leqslant N-1, \frac{n}{k_1\ldots k_{N-1}}>\lfloor a_n\rfloor\right\}.
$$

Now let $T$ be function defined in \eqref{eq:T-definition}, and~let
$$ 
\varphi_n(k_1,\ldots,k_{N-1})=T\Big(k_1,\ldots,k_{N-1},\frac{n}{k_1\ldots k_{N-1}}\Big).
$$
For $\mathbf{k}\in\mathcal{C}_n$, according to the mean value theorem, we have
\begin{adjustwidth}{-\extralength}{0cm}
\centering 
\begin{align}\label{afii}
    T(\mathbf{k})-\varphi_n(k_1,\ldots,k_{N-1})&=T\left(k_1,\ldots,k_{N-1},\Big\lceil\frac{n}{k_1\ldots k_{N-1}}\Big\rceil\right)-T\Big(k_1,\ldots,k_{N-1},\frac{n}{k_1\ldots k_{N-1}}\Big)\nonumber\\
    &=\frac{\partial T(k_1,\dots,k_{N-1},t)}{\partial t}\bigg{|}_{t=t^*}
    \left(\Big\lceil\frac{n}{k_1\ldots k_{N-1}}\Big\rceil-\frac{n}{k_1\ldots k_{N-1}}\right),
\end{align}
\end{adjustwidth}
where 
$$
\frac{\partial T(k_1,\dots,k_{N-1},t)}{\partial t}=\log\lambda_N-\log t-\frac{1}{2t},
$$
and 
$$
t^*\in\left(\frac{n}{k_1\ldots k_{N-1}},\Big\lceil\frac{n}{k_1\ldots k_{N-1}}\Big\rceil\right).
$$
According to the restrictions $k_i>\lfloor a_n\rfloor$, $1\leqslant i\leqslant N-1$, we have that
$$
\Big\lceil\frac{n}{k_1\ldots k_{N-1}}\Big\rceil\leqslant 2n^{1/N}(\log n)^{N-1}
$$
if $n$ is sufficiently large. On~the other hand, for~large $n$
$$
\frac{n}{k_1\ldots k_{N-1}}>\lfloor a_n\rfloor\geqslant \frac12\frac{n^{1/N}}{\log n}.
$$
Therefore, by~the mean value equality \eqref{afii}, we have that  
$$
T(\mathbf{k})-\varphi_n(k_1,\ldots,k_{N-1})=O(\log n),
$$
and by equalities \eqref{po1} and \eqref{eq:T-definition}
\begin{align}\label{opa2}
\mathcal{P}_n^{(N)}\mathop{\sim}\limits_{n\rightarrow\infty}\mathcal{P}_{C,n}=\sum_{\widehat{\mathbf{k}}\in\mathcal{G}_n}\exp\big\{\varphi_n(\widehat{\mathbf{k}})+O(\log n)\big\},
\end{align}
where
\begin{align}\label{fa1}
\mathcal{G}_n=\left\{\widehat{\mathbf{k}}=(k_1,\ldots,k_{N-1}): k_i>\lfloor a_n\rfloor, 1\leqslant i\leqslant N-1, \frac{n}{k_1\ldots k_{N-1}}>\lfloor a_n\rfloor \right\},
\end{align}
and
\begin{align}\label{fa2}    \varphi_n(\widehat{\mathbf{k}})&={-\sum_{i=1}^N\lambda_i}+\sum_{i=1}^{N-1}k_i\big(\log\lambda_i-\log k_i+1\big)-\frac12\sum_{i=1}^{N-1}\log 2\pi k_i\nonumber\\
    &+\frac{n}{k_1\ldots k_{N-1}}\Big(\log\lambda_N-\log\frac{n}{k_1\ldots k_{N-1} }+1\Big)-\frac12\log\frac{2\pi n}{k_1\ldots k_{N-1}}.
\end{align}

\subsection{Passage from the First-Layer Sum to the Corresponding~Integral}

Set
\[
S_n
:=
\sum_{\widehat{\mathbf{k}}\in \mathcal{G}_n}
\exp\!\bigl(\varphi_n(\widehat{\,\mathbf{k}})\bigr),
\quad
I_n
:=
\int_{\widetilde{\mathcal{G}}_n} \exp\!\bigl(\varphi_n(\mathbf{x})\bigr)\,{\rm d}\mathbf{x},
\]
{where $\mathcal{G}_n$ defined in \eqref{fa1}, function $\varphi_n$ defined in \eqref{fa2}, and~}
\begin{equation}\label{paa}
\widetilde{\mathcal{G}}_n
:=
\left\{
\mathbf{x}=(x_1,\dots,x_{N-1})\in (a_n,\infty)^{N-1}
:
\frac{n}{x_1\ldots x_{N-1}}>a_n
\right\}.
\end{equation}
{In this subsection, we prove that }
\begin{equation}\label{par}
S_n\mathop{\sim}\limits_{n\rightarrow\infty} I_n.
\end{equation}

Let \(\widehat{\mathbf{k}}^{ *}=(k_1^*,\dots,k_{N-1}^*)\) be the maximizer of the function \(\varphi_n\) on the area \(\widetilde{\mathcal{G}}_n\), and~let 
$
e_n=n^{1/(2N)}\log n.
$
Define
\[
\mathcal{E}_n
:=
\left\{
\mathbf{x}=(x_1,\ldots,x_{N-1})\in \widetilde{\mathcal{G}}_n:
\max_{1\leqslant i\leqslant N-1}|x_i-k_i^*|\leqslant  e_n
\right\}.
\]
It is clear that 
\[
S_n=S_n^{\mathrm{loc}}+S_n^{\mathrm{rem}},
\quad
I_n=I_n^{\mathrm{loc}}+I_n^{\mathrm{rem}},
\]
where the local parts are restricted to \(\mathcal{E}_n\), and~the remainder parts to its~complement.

{
 \begin{enumerate}[label=(i), leftmargin=*, labelsep=*]
\item First, we show that the remainder terms are negligible. Since
\(\lfloor a_n\rfloor\geqslant a_n/2\) for all sufficiently large \(n\), introduce
 \end{enumerate}
\[
\Omega_n
:=
\left\{
\mathbf{x}=(x_1,\dots,x_{N-1})\in (a_n/2,\infty)^{N-1}:
\frac{n}{x_1\ldots x_{N-1}}>a_n/2
\right\}.
\]
Then
\[
\mathcal G_n\subset\Omega_n,
\quad
\widetilde{\mathcal G}_n\subset\Omega_n .
\]
Hence, for~every
\(i=1,\ldots,N-1\),
\[
x_i
<
\frac{n}{(a_n/2)^{N-1}}
=
2^{N-1}n^{1/N}(\log n)^{N-1}.
\]
Consequently,
\[
\#\mathcal G_n
\leqslant
\left(2^{N-1}n^{1/N}(\log n)^{N-1}\right)^{N-1},
\quad
|\widetilde{\mathcal G}_n|
\leqslant
\left(2^{N-1}n^{1/N}(\log n)^{N-1}\right)^{N-1},
\]
where the symbol $|A|$ denotes the Lebesgue measure of a set $A$.

We show that, for~all sufficiently large \(n\),
\begin{align}\label{poo}
\varphi_n(\mathbf{x})-\varphi_n(\widehat{\mathbf{k}}^*)
\leqslant
-\frac{1}{24}\log^2 n,
\ \ {\rm for}\ \ 
\mathbf{x}\in\widetilde{\mathcal G}_n\setminus\mathcal E_n . 
\end{align}
Indeed, put
\[
x_N:=\frac{n}{x_1\ldots x_{N-1}},
\quad
k_N^*:=\frac{n}{k_1^*\ldots k_{N-1}^*},
\]
Then, by~expression \eqref{fa2},
\[
\varphi_n(\mathbf{x})
=
C_n
+
\sum_{i=1}^N
x_i(\log\lambda_i-\log x_i+1),
\]
and, similarly,
\[
\varphi_n(\widehat{\mathbf{k}}^*)
=
C_n
+
\sum_{i=1}^N
k_i^*(\log\lambda_i-\log k_i^*+1),
\]
where the quantity \(C_n\) does not depend on \(\mathbf{x}\).

By the Lagrange Equation \eqref{poi} we have
\[
\log\lambda_i-\log k_i^*-\frac{1}{2k_i^*}
-\alpha^*\frac{n}{k_i^*}=0,
\quad i=1,\dots,N.
\]
Multiplying by \(k_i^*\), we obtain
\[
k_i^*\log\frac{\lambda_i}{k_i^*}
=
\alpha^*n+\frac12
=:\beta_n,
\quad i=1,\ldots,N.
\]
By \eqref{eq:alpha-final-general}, we have that 
 \(\beta_n<0\) for all sufficiently large \(n\).
Now denote
\[
q_i:=\frac{x_i}{k_i^*},\quad i=1,\dots,N.
\]
Using \(x_i=k_i^*q_i\), we get
\[
\begin{aligned}
\varphi_n(\mathbf{x})-\varphi_n(\widehat{\mathbf{k}}^*)
&=
\sum_{i=1}^N
k_i^*
\left(
(q_i-1)\Big(\log\frac{\lambda_i}{k_i^*}+1\Big)
-q_i\log q_i
\right)  \\
&=
\beta_n\sum_{i=1}^N(q_i-1)
-
\sum_{i=1}^N
k_i^*
\left(q_i\log q_i-q_i+1\right).
\end{aligned}
\]
Since \(q_1\ldots q_N=1\), the~classical AM-GM inequality gives
\[
\sum_{i=1}^Nq_i\geqslant N.
\]
Hence, from~\(\beta_n<0\), it follows that
\begin{align}\label{do}
\varphi_n(\mathbf{x})-\varphi_n(\widehat{\mathbf{k}}^*)
\leqslant
-
\sum_{i=1}^N
k_i^*
\left(q_i\log q_i-q_i+1\right)
\end{align}
for sufficiently large $n$.

Let now temporarily
$
h(q):=q\log q-q+1, q>0,
$ and  \(\mathbf{x}\in \widetilde{\mathcal G}_n\setminus\mathcal E_n\). By~the definition of
\(\mathcal E_n\), there exists an index \(j\in\{1,\ldots,N-1\}\) such that
\[
|x_j-k_j^*|>e_n .
\]
Using \eqref{eq:kstar-final-general}, we have that
\[
\frac12 n^{1/N}\leqslant k_j^*\leqslant 2n^{1/N}.
\]
for all sufficiently large $n$. 
Therefore, the~chosen coordinate $j$ satisfies the estimate
\[
|q_j-1|
=
\frac{|x_j-k_j^*|}{k_j^*}
>
\frac{e_n}{2n^{1/N}}=\frac{\log n}{2 n^{1/(2N)}}.
\]
Let
\[
\varepsilon_n:=\frac{\log n}{2n^{1/(2N)}}.
\]
Since \(\varepsilon_n\to0\), we have \(\varepsilon_n<1/2\) for all
sufficiently large \(n\). Moreover,
\[
h'(q)=\log q,
\]
so \(h\) is decreasing on \((0,1)\) and increasing on \((1,\infty)\).
Hence, from~\(|q_j-1|>\varepsilon_n\),
\[
h(q_j)
\geqslant
\min\{h(1-\varepsilon_n),h(1+\varepsilon_n)\}.
\]
Since
\[
h(1\pm\varepsilon)
=
\frac{\varepsilon^2}{2}+O(\varepsilon^3),
\qquad \varepsilon\to0,
\]
we obtain, for~all sufficiently large \(n\),
\[
h(q_j)
\geqslant
\frac{\varepsilon_n^2}{3}
=
\frac{\log^2n}{12n^{1/N}}.
\]
Therefore, using \eqref{do} and
\(k_j^*\geqslant\frac12n^{1/N}\),
\[
\varphi_n(\mathbf{x})-\varphi_n(\widehat{\mathbf{k}}^*)
\leqslant
-k_j^*h(q_j)
\leqslant
-\frac1{24}\log^2n,
\]
which proves \eqref{poo}.

Moreover, if~\(\widehat{\mathbf{k}}\in\mathcal G_n\) satisfies
\[
\max_{1\leqslant i\leqslant N-1}|k_i-k_i^*|\leqslant e_n,
\]
then, by~\eqref{eq:kstar-final-general} and
\(e_n=o(n^{1/N})\),
\[
k_i=n^{1/N}(1+o(1)),
\qquad
\frac{n}{k_1\ldots k_{N-1}}
=
n^{1/N}(1+o(1))
>a_n
\]
for all sufficiently large \(n\).
Hence \(\widehat{\mathbf{k}}\in\mathcal E_n\).
Therefore every
\(\widehat{\mathbf{k}}\in\mathcal G_n\setminus\mathcal E_n\)
satisfies
\[
\max_{1\leqslant i\leqslant N-1}|k_i-k_i^*|>e_n,
\]
and the same argument as above gives \eqref{poo} for the discrete remainder as~well.

Hence, by~the estimates for \(\#\mathcal G_n\) and \(|\widetilde{\mathcal G}_n|\),
\[
S_n^{\mathrm{rem}}
\leqslant
\left(2^{N-1}n^{1/N}(\log n)^{N-1}\right)^{N-1}
\exp\left\{
\varphi_n(\widehat{\mathbf{k}}^*)-\frac{1}{24}\log^2 n
\right\},
\]
and
\[
I_n^{\mathrm{rem}}
\leqslant
\left(2^{N-1}n^{1/N}(\log n)^{N-1}\right)^{N-1}
\exp\left\{
\varphi_n(\widehat{\mathbf{k}}^*)-\frac{1}{24}\log^2 n
\right\}.
\]
On the other hand, by~the estimate \eqref{eq:local-gradient} obtained below,
\[
S_n^{\mathrm{loc}}
\geqslant
\exp\{\varphi_n(\widehat{\mathbf{k}}^*)+o(1)\},
\qquad
I_n^{\mathrm{loc}}
\geqslant
\exp\{\varphi_n(\widehat{\mathbf{k}}^*)+o(1)\}.
\]
Therefore,
\[
\frac{S_n^{\mathrm{rem}}}{S_n^{\mathrm{loc}}}
\leqslant
\left(2^{N-1}n^{1/N}(\log n)^{N-1}\right)^{N-1}
\exp\left\{
-\frac{1}{24}\log^2 n+o(1)
\right\}
\mathop{\to}\limits_{n\rightarrow\infty} 0,
\]
and
\[
\frac{I_n^{\mathrm{rem}}}{I_n^{\mathrm{loc}}}
\leqslant
\left(2^{N-1}n^{1/N}(\log n)^{N-1}\right)^{N-1}
\exp\left\{
-\frac{1}{24}\log^2 n+o(1)
\right\}
\mathop{\to}\limits_{n\rightarrow\infty} 0.
\]
Consequently,
\begin{align}\label{opa1}
S_n^{\mathrm{rem}}=o\!\left(S_n^{\mathrm{loc}}\right),
\quad
I_n^{\mathrm{rem}}=o\!\left(I_n^{\mathrm{loc}}\right),
\end{align}
}
as $n\rightarrow\infty$.

{\begin{enumerate}[label=(ii), leftmargin=*, labelsep=*]
\item It remains to compare the local parts. To this end, we estimate 
\end{enumerate}

$$
\sup_{\mathbf{x}\in\mathcal{E}_n}\|\nabla\varphi_n(\mathbf{x})\|=\sup_{\mathbf{x}\in\mathcal{E}_n}\sqrt{\sum_{l=1}^{N-1}\left(\frac{\partial{\varphi_n}(\mathbf{x})}{\partial\, x_l}\right)^2}
$$
Let us apply the first-order multivariate Taylor expansion at the point
$
\widehat{\mathbf{k}}^{ *}.
$
For any $1\leqslant l\leqslant N-1$, we have

\[
\frac{\partial \varphi_n}{\partial x_l}(\mathbf{x})
=
\sum_{i=1}^{N-1}
\frac{\partial^2\varphi_n(\ \widehat{\mathbf{k}}^{*})}{\partial x_l\partial x_i}(x_i-k_i^*)
+\frac12\sum_{i=1}^{N-1}\sum_{j=1}^{N-1}\frac{\partial^3\varphi_n \big(\,\widehat{\mathbf{k}^*}+\theta({\mathbf{x}}-\widehat{\mathbf{k}^*})\big)}{\partial x_l\partial x_i \partial x_j}(x_i-k_i^*)(x_j-k_j^*)
\]
with some $\theta\in (0,1)$. According to relation \eqref{eq:kstar-final-general} we get that 
$$
\frac{\partial^2 \varphi_n}{\partial x_i\partial x_j}(\,\widehat{ \mathbf{k}}^*)
=
O\!\left(\frac{\log n}{n^{1/N}}\right)
$$
for any pair of indices $1\leqslant i,j\leqslant N-1$. Therefore,
\begin{align}\label{oo-1}
\sup_{\mathbf{x}\in \mathcal{E}_n}\|\nabla\varphi_n(\mathbf{x})\|
&=
O\!\left(\frac{e_n\log n}{n^{1/N}}\right)\nonumber\\&+O\!\left(e_n^2\max_{1\leqslant l\leqslant N-1}\sup_{\mathbf{x}\in\mathcal{E}_n}\left|\sum_{i=1}^{N-1}\sum_{j=1}^{N-1}\frac{\partial^3\varphi_n \big(\,\widehat{\mathbf{k}}^*+\theta({\mathbf{x}}-\widehat{\mathbf{k}}^*)\big)}{\partial x_l\partial x_i \partial x_j}\right|\,\right).
\end{align}
Since \(e_n=o(n^{1/N})\) and \(k_j^*=n^{1/N}(1+o(1))\), coordinates of every \(\mathbf{x}=(x_1,\ldots, x_{N-1})\in \mathcal{E}_n\) satisfy
\[
x_i=O(n^{1/N}),
\quad
x_i^{-1}=O(n^{-1/N}),
\quad 1\leqslant i\leqslant N-1.
\]
In addition,  
\[
\frac{n}{x_1\ldots x_{N-1}}=O(n^{1/N}),
\quad
\log\!\left(\frac{n}{x_1\ldots x_{N-1}}\right)=O(\log n).
\]
These estimates and the direct inspection of the explicit formulas for the derivatives of the function $\varphi_n$ give that 
$$
\frac{\partial^3\varphi_n \big(\,\widehat{\mathbf{k}}^*+\theta({\mathbf{x}}-\widehat{\mathbf{k}}^*)\big)}{\partial x_l\partial x_i \partial x_j}=O\!\left(\frac{\log n}{n^{2/N}}\right)
$$
for any indices $\{l,i,j\}\in\{1,2,\ldots, N-1\}$, and~some ``constant'' in symbol $O$ depending only on $N$. Consequently, by~using \eqref{oo-1} we get
\begin{equation}\label{eq:local-gradient}
\sup_{\mathbf{x}\in \mathcal{E}_n}\|\nabla\varphi_n(\mathbf{x})\|=o(1).
\end{equation}
Therefore, throughout \(\mathcal{E}_n\), the~gradient \(\nabla\varphi_n\) is \(o(1)\). Hence, moving by one discrete step in any coordinate direction changes the value of \(\varphi_n\) by only \(o(1)\). Thus, the~local discrete sum is a Riemann sum for the local integral
\[
S_n^{\mathrm{loc}}
=
I_n^{\mathrm{loc}}\bigl(1+o(1)\bigr).
\]
Combining this with the negligibility of the remainder terms \eqref{opa1} yields the desired relation~\eqref{par}.

\subsection{Integral Expression for the Main~Probability }

{By the reduction to \(N-1\) dimensions on the first balanced layer (see \eqref{opa2}) we derived~that }
\[
\mathcal P_{C,n}
=
\sum_{\widehat{\mathbf{k}}\in \mathcal{G}_n}\exp\bigl\{\varphi_n(\widehat{\mathbf{k}})+O(\log n)\bigr\}.
\]
{Therefore, according to \eqref{par}}
\[
\mathcal{P}_{C,n}
=
\exp\bigl\{O(\log n)\bigr\}S_n
=
\bigl(1+o(1)\bigr)\exp\bigl\{O(\log n)\bigr\} I_n.
\]
Now we focus on the integral \(I_n\). We have that 
\begin{align*}
I_n
=
\int_{\widetilde{\mathcal{G}}_n}
\exp\Bigg\{
&-\sum_{i=1}^N \lambda_i
+\sum_{i=1}^{N-1} x_i\bigl(\log\lambda_i-\log x_i+1\bigr) \\
&\quad
+\frac{n}{x_1\ldots x_{N-1}}
\left(
\log\lambda_N-\log\frac{n}{x_1\ldots x_{N-1}}+1
\right) \\
&\quad
-\frac12\sum_{i=1}^{N-1}\log(2\pi x_i)
-\frac12\log\!\left(2\pi\frac{n}{x_1\ldots x_{N-1}}\right)
\Bigg\}
\,{\rm d}x_1\ldots {\rm d}x_{N-1}.
\end{align*}
For our result, it is enough to retain the exponent up to an \(O(\log n)\) term. Since
\[
\sum_{i=1}^{N-1}\log(2\pi x_i)
+\log\!\left(2\pi\frac{n}{x_1\ldots x_{N-1}}\right)
=
{N}\log(2\pi)+\log n
\]
for any $\mathbf{x}=(x_1,\ldots,x_{N-1})\in \widetilde{\mathcal{G}}_n$,  these logarithmic terms contribute only \(O(\log n)\) to the exponent and may be absorbed into the separate exponential error already presented above. Accordingly, we introduce a function
\[
\psi_n(\mathbf{x})
=
-\sum_{i=1}^N \lambda_i
+\sum_{i=1}^{N-1} x_i\bigl(\log\lambda_i-\log x_i+1\bigr)
+\frac{n}{x_1\ldots x_{N-1}}
\left(
\log\lambda_N-\log\frac{n}{x_1\ldots x_{N-1}}+1
\right),
\]
{implying that
\[
I_n
=
\widetilde{{I}}_n\exp\bigl\{O(\log n)\bigr\},
\]
and
\begin{equation}\label{par2}
\mathcal{P}_{C,n}=\widetilde{I}_n\exp\big\{O(\log n)\big\},
\end{equation}
where
\[
\widetilde{I}_n=\int_{\widetilde{\mathcal{G}}_n}
\exp\bigl\{\psi_n(\mathbf{x})\bigr\}\,{\rm d}x_1\ldots {\rm  d}x_{N-1}.
\]}

{
\subsection{Multidimensional Laplace~Approximation}

In this subsection, we obtain asymptotics of the integral
$\widetilde{I}_n$. For~this, we use the multidimensional
Laplace approximation, presented in the lemma below
{and discussed by Wong in Chapter IX,
Section~5, pp.~494--500 of~\cite{w-2001}.}

\begin{Lemma}\label{lem:wong-multidimensional}
Let
\[
J(\lambda)
=
\int_D g(\mathbf{x})\exp\{-\lambda f(\mathbf{x})\}\,{\rm d}\mathbf{x},
\]
where \(\lambda\) is a large positive parameter, \(D\) is a possibly
unbounded domain in \(\mathbb{R}^d\), $\mathbf{x}_0\in\mathbb{R}^d$, and~\(f\) and \(g\) are infinitely
differentiable in \(D\). Suppose that the following conditions hold:

\begin{enumerate}[label={\rm (\roman*)},leftmargin=2.2em,labelsep=8mm]
\item
The integral \(J(\lambda)\) converges absolutely for all
\(\lambda\geqslant\lambda_0\);

\item
For any \(\varepsilon>0\),
\[
\inf\left\{
f(\mathbf{x})-f(\mathbf{x}_0):
\mathbf{x}\in D,\ |\mathbf{x}-\mathbf{x}_0|\geqslant\varepsilon
\right\}>0;
\]

\item
The Hessian matrix
\[
A
=
\left.
\left(
\frac{\partial^2 f}{\partial x_i\partial x_j}
\right)
\right|_{\mathbf{x}=\mathbf{x}_0}
\]
is positively definite.
\end{enumerate}
If \(\mathbf{x}_0\) is an interior point of \(D\), then
\[
J(\lambda)
\mathop{\sim}\limits_{\lambda\rightarrow\infty}
{\rm e}^{-\lambda f(\mathbf{x}_0)}
\sum_{k=0}^{\infty}
C_k\lambda^{-d/2-k},
\]
where the \(C_k\) are constants. In~particular,
\[
J(\lambda)
\mathop{\sim}\limits_{\lambda\rightarrow\infty}
\left(\frac{2\pi}{\lambda}\right)^{d/2}
g(\mathbf{x}_0)(\det A)^{-1/2}
\exp\{-\lambda f(\mathbf{x}_0)\}.
\]
\end{Lemma}

\medskip

To use Lemma \ref{lem:wong-multidimensional}, we rewrite integral \(\widetilde{I}_n\) by changing variables according to the formula 
\[
x_i={n}^{1/N} z_i,
\quad
i=1,\dots,N-1.
\]
After this change the integration region $\widetilde{\mathcal{G}}_n$ transforms into  region 
\[
\mathcal{D}_n
=
\left\{
{\mathbf z}=(z_1,\dots,z_{N-1})
\in
\left(\frac{1}{\log n},\infty\right)^{N-1}
:
\prod_{i=1}^{N-1}z_i<\log n
\right\}.
\]
In addition,
\begin{align*}
    \psi_n(\mathbf{x})=\psi(n^{1/N}\mathbf{z})&=
-\sum_{i=1}^N\lambda_i
-
\frac{n^{1/N}\log n}{N}
\sum_{i=1}^{N-1} z_i
+
n^{1/N}
\sum_{i=1}^{N-1}
z_i
\bigl(
1+\log\lambda_i-\log z_i
\bigr)\\
&\quad -\frac{n^{1/N}\log n}{N\,z_1\ldots z_{N-1}}+\frac{n^{1/N}}{z_1\ldots z_{N-1}}\left(1+\log \lambda_N-\log\frac{1}{z_1\ldots z_{N-1}}\right).
\end{align*}
Therefore, after~the change of variables
\begin{align}\label{aa-1}
\widetilde{I}_n=
n^{(N-1)/N}
{\rm e}^{-\sum_{i=1}^N\lambda_i}
\int_{\mathcal{D}_n}
\exp\left\{
-\frac{n^{1/N}\log n}{N}S({\mathbf z})
+
n^{1/N}Q({\mathbf z})
\right\}
\,{\rm d}\mathbf{z},
\end{align}
where 
$$
S({\mathbf z})
=
z_1+\ldots+z_{N-1}+z_{N},\quad
Q({\mathbf z})
=
\sum_{i=1}^{N}
z_i
\bigl(
1+\log\lambda_i-\log z_i
\bigr)
$$
with $z_N=1/(z_1\ldots z_{N-1})$.

Recall from \eqref{eq:kstar-final-general} that
\[
k_i^*
=
n^{1/N}(1+o(1)).
\]
Since \(x_i=n^{1/N} z_i\), the~corresponding
rescaled coordinates satisfy
\[
z_i^*
=
\frac{k_i^*}{n^{1/N}}
=
1+o(1)
\]
as $n\rightarrow\infty$. 
Thus, in~the rescaled coordinates, the~relevant point is
\(
{\mathbf z}_0=(1,\ldots,1).
\)
Since
\(
z_1\ldots z_N=1,
\)
the AM/GM inequality yields that
\[
S({\mathbf z})
=
\sum_{i=1}^N z_i
\geqslant N
=
S(\mathbf z_0).
\]
A direct differentiation gives

\begin{equation}\label{eq:S-hessian}
\nabla S({\mathbf z}_0)=\left.\Big(\frac{\partial S}{\partial z_1},\ldots\frac{\partial S}{\partial z_{N-1}}\Big)\right|_{\mathbf{z}=\mathbf{z}_0}=\mathbf{0},\ \ 
\left(
\frac{\partial^2 S(\mathbf{z}_0)}{\partial z_i\partial z_j}
\right)_{i,j=1}^{N-1}=\begin{pmatrix}
    2&1&1&\ldots&1\\
    1&2&1&\ldots&1\\
    1&1&2&\ldots& 1\\
    &&&\ldots&\\
    1&1&1&\ldots&2
\end{pmatrix}.
\end{equation}
{This Hessian matrix is positive definite. Hence, there exist constants
\(c_*,c^*,\delta>0\), independent of \(n\), such that
\begin{equation}\label{eq:S-local}
c_*\left\|{\mathbf z}-{\mathbf z}_0\right\|^2
\leqslant
S({\mathbf z})-N
\leqslant
c^*\left\|{\mathbf z}-{\mathbf z}_0\right\|^2.
\end{equation}
whenever}
\(
\left\|{\mathbf z}-{\mathbf z}_0\right\|
\leqslant\delta.
\)
For all sufficiently large \(n\), \[ \left\{{\mathbf z}: \left\|{\mathbf z}-{\mathbf z}_0\right\| \leqslant\delta \right\} \subset {\mathcal{D}_n}. \] Moreover \(S\) has its unique minimum at \({\mathbf z}_0=(1,\ldots,1)\) and tends to infinity at the boundary of \(\mathcal D=[0,\infty)^{N-1}\) and as \(\|{\mathbf z}\|\to\infty\). Hence, there exists \(\eta=\eta(\delta)>0\) such that
\begin{equation}\label{eq:S-gap} 
S({\mathbf z})\geqslant N+\eta \end{equation} whenever \({\mathbf z}\in\mathcal D\) and\( \left\|{\mathbf z}-{\mathbf z}_0\right\| \geqslant\delta. \)
In addition, by~the mean value theorem, there exists a constant \(C_Q>0\) such that
\begin{equation}\label{eq:Q-local}
\left|
Q({\mathbf z})-Q({\mathbf z}_0)
\right|
\leqslant
C_Q\left\|{\mathbf z}-{\mathbf z}_0\right\|.
\end{equation}
if 
\(
\left\|{\mathbf z}-{\mathbf z}_0\right\|
\leqslant\delta.
\)
Moreover,
\[
Q({\mathbf z}_0)
=
N+\log\prod_{i=1}^N\lambda_i,
\]
and
\begin{align}\label{poi-2}
Q({\mathbf z})
\leqslant
\sum_{i=1}^N\lambda_i,
\quad
{\mathbf z}\in\mathcal D
\end{align}
because
\[
z(1+\log\lambda-\log z)\leqslant\lambda
\]
for every \(z>0\) and \(\lambda>0\). 

Inside of the neighborhood \(\|\mathbf{z}-\mathbf{z}_0\|\leqslant \delta \),
by \eqref{eq:S-local} and
\eqref{eq:Q-local}, we have
\[
\begin{aligned}
&
-\frac{n^{1/N}\log n}{N}
\bigl(S({\mathbf z})-N\bigr)
+
n^{1/N}
\bigl(
Q({\mathbf z})-Q({\mathbf z}_0)
\bigr)
\\
&\quad
\leqslant
-c_*\,\frac{n^{1/N}\log n}{N}
\left\|{\mathbf z}-{\mathbf z}_0\right\|^2
+
C_Q\,n^{1/N}
\left\|{\mathbf z}-{\mathbf z}_0\right\|.
\end{aligned}
\]
Denoting
\(
r=\left\|{\mathbf z}-{\mathbf z}_0\right\|,
\)
and using inequality 
\[
-ar^2+br
=
-a\left(r-\frac{b}{2a}\right)^2+\frac{b^2}{4a}
\leqslant
\frac{b^2}{4a},
\]
with
\[
a=c_*\,\frac{n^{1/N}\log n}{N},
\qquad
b=C_Q\,n^{1/N},
\]
we obtain
\[
-c_*\,\frac{n^{1/N}\log n}{N}\,r^{\,2}
+
C_Q\,n^{1/N}r
\leqslant
\frac{C_Q^2N}{4c_*}\,
\frac{n^{1/N}}{\log n}.
\]
Consequently,
\begin{align}\label{poi-1}
&
-\frac{n^{1/N}\log n}{N}S({\mathbf z})
+
n^{1/N}Q({\mathbf z})\nonumber
\\
&\quad\leqslant
-n^{1/N}\log n
+
n^{1/N}Q({\mathbf z}_0)
+
\frac{C_Q^2N}{4c_*}\,
\frac{n^{1/N}}{\log n}.
\end{align}
Outside of the  neighborhood $\|\mathbf{z}-\mathbf{z}_0\|\leqslant \delta$ by \eqref{eq:S-gap} and \eqref{poi-2}, we have that
\begin{align}\label{poi-3}
&
-\frac{n^{1/N}\log n}{N}S({\mathbf z})
+
n^{1/N}Q({\mathbf z})
\nonumber\\
&\quad\leqslant
-n^{1/N}\log n
-\frac{\eta}{N}n^{1/N}\log n
+
n^{1/N}\sum_{i=1}^N\lambda_i.
\end{align}
The derived  estimates \eqref{poi-1} and \eqref{poi-3} imply that for all
sufficiently large \(n\)
\begin{align}\label{poi-5}
&
\exp\left\{
-\frac{n^{1/N}\log n}{N}S({\mathbf z})
+n^{1/N}Q({\mathbf z})
\right\}
\nonumber\\
&\quad\leqslant
\exp\left\{
-n^{1/N}\log n
+n^{1/N}Q({\mathbf z}_0)
+\frac{C_Q^2N}{4c_*}\frac{n^{1/N}}{\log n}
\right\}
\end{align}
throughout \(\mathcal{D}_n\).

For the corresponding lower bound, let us consider the ball 
\(
\left\|{\mathbf z}-{\mathbf z}_0\right\|
\leqslant{1}/{\log n}.
\)
This ball is contained in ${\mathcal{D}_n}$ for all sufficiently large \(n\), and~on it
\[
S({\mathbf z})-N
\leqslant \frac{c^*}{\log^2 n},
\quad
Q({\mathbf z})-Q({\mathbf z}_0)
\leqslant 
\frac{C_Q}{\log n}
\]
due to estimates \eqref{eq:S-local},
\eqref{eq:Q-local}.
Consequently, for~all $\mathbf{z}\in\mathcal{D}_n$
\begin{align}\label{poi-4}
&
-\frac{n^{1/N}\log n}{N}S({\mathbf z})
+
n^{1/N}Q({\mathbf z})
\nonumber\\
&\quad\geqslant
-n^{1/N}\log n
+
n^{1/N}Q({\mathbf z}_0)
-
\left(\frac{c^*}{N}+C_Q\right)
\frac{n^{1/N}}{\log n}.
\end{align}
Definition of \(\mathcal{D}_n\) implies that 
\(
z_i<(\log n)^{N-1}\) for any index \(i\in\{1,\ldots,N-1\}\). Hence,
\[
|\mathcal{D}_n|\leqslant(\log n)^{(N-1)^2}.
\]
\enlargethispage{1cm}
By this inequality and the upper estimate \eqref{poi-5} provided for sufficiently large $n$, we get that 
\begin{align*}
&\int_{\mathcal{D}_n}
\exp\left\{
-\frac{n^{1/N}\log n}{N}S({\mathbf z})
+n^{1/N}Q({\mathbf z})
\right\}
\,{\rm d}\mathbf{z}
\\
&\quad\leqslant
\int_{\mathcal{D}_n}
\exp\left\{
-n^{1/N}\log n
+n^{1/N}Q({\mathbf z}_0)
+\frac{C_Q^2N}{4c_*}\frac{n^{1/N}}{\log n}
\right\}
\,{\rm d}\mathbf{z}
\\
&\quad=
|\mathcal{D}_n|
\exp\left\{
-n^{1/N}\log n
+n^{1/N}Q({\mathbf z}_0)
+\frac{C_Q^2N}{4c_*}\frac{n^{1/N}}{\log n}
\right\}
\\
&\quad\leqslant
(\log n)^{(N-1)^2}
\exp\left\{
-n^{1/N}\log n
+n^{1/N}Q({\mathbf z}_0)
+\frac{C_Q^2N}{4c_*}\frac{n^{1/N}}{\log n}
\right\}.
\end{align*}
On the other hand, using the
lower estimate obtained in \eqref{poi-4}, we get that 
\begin{align*}
&\int_{\mathcal{D}_n}
\exp\left\{
-\frac{n^{1/N}\log n}{N}S({\mathbf z})
+n^{1/N}Q({\mathbf z})
\right\}
\,{\rm d}\mathbf{z}
\\
&\quad\geqslant
\int\limits_{\left\|{\mathbf z}-{\mathbf z}_0\right\|
\leqslant1/\log n}
\exp\left\{
-\frac{n^{1/N}\log n}{N}S({\mathbf z})
+n^{1/N}Q({\mathbf z})
\right\}
\,{\rm d}\mathbf{z}
\\
&\quad\geqslant
\frac{\pi^{(N-1)/2}}
{\Gamma\left(\frac{N+1}{2}\right)}
\frac1{(\log n)^{N-1}}
\\
&\qquad\times
\exp\left\{
-n^{1/N}\log n
+n^{1/N}Q({\mathbf z}_0)
-\left(\frac{c^*}{N}+C_Q\right)
\frac{n^{1/N}}{\log n}
\right\}.
\end{align*}
Since
\begin{align*}
(\log n)^{(N-1)^2}
\leqslant
\exp\bigl\{c_3\log\log n\bigr\},
\quad \frac{\pi^{(N-1)/2}}
{\Gamma\left(\frac{N+1}{2}\right)}
\frac1{(\log n)^{N-1}}
\geqslant
\exp\bigl\{-c_4\log\log n\bigr\},
\end{align*}
for some positive quantities $c_3$, $c_4$ not depending on $n$,
and 
\(
\log\log n
=
o\!\left({n^{1/N}}/{\log n}\right),
\)
the preceding two  inequalities  yield
\begin{align}\label{poi-6}
&\int_{\mathcal{D}_n}
\exp\left\{
-\frac{n^{1/N}\log n}{N}S({\mathbf z})
+n^{1/N}Q({\mathbf z})
\right\}
\,{\rm d}\mathbf{z}
\nonumber\\
&\quad=
\exp\!\left\{
-n^{1/N}\log n
+n^{1/N}Q({\mathbf z}_0)
+
O\!\left(\frac{n^{1/N}}{\log n}\right)
\right\}.
\end{align}
Similarly, the~estimate  \(S({\mathbf z})\geqslant N\) implies that
\[
\begin{aligned}
&\int_{\mathcal{D}_n}
\exp\left\{
-\frac{n^{1/N}\log n}{N}S({\mathbf z})
\right\}
\,{\rm d}\mathbf{z}
\\
&\quad\leqslant
(\log n)^{(N-1)^2}
\exp\left\{-n^{1/N}\log n\right\},
\end{aligned}
\]
and the estimate 
\[
S(\widehat{\mathbf z})-N
\leqslant
\frac{c^*}{\log^2 n}
\]
provided if \(
\|{\mathbf z}-{\mathbf z}_0\|
\leqslant1/\log n
\),
gives that
\begin{align}\label{poi-9}
&\int_{\mathcal{D}_n}
\exp\left\{
-\frac{n^{1/N}\log n}{N}S({\mathbf z})
\right\}
\,{\rm d}\mathbf{z}
\nonumber\\
&\quad\geqslant
\frac{\pi^{(N-1)/2}}
{\Gamma\!\left(\frac{N+1}{2}\right)}
\frac1{(\log n)^{N-1}}
\exp\left\{
-n^{1/N}\log n
-\frac{c^*}{N}\frac{n^{1/N}}{\log n}
\right\}.
\end{align}
Therefore,
\begin{align}\label{poi-7}
\int_{\mathcal{D}_n}
\exp\left\{
-\frac{n^{1/N}\log n}{N}S({\mathbf z})
\right\}
\,{\rm d}\mathbf{z}
=
\exp\left\{
-n^{1/N}\log n
+
O\!\left(\frac{n^{1/N}}{\log n}\right)
\right\}.
\end{align}
Dividing the  estimates \eqref{poi-6} and \eqref{poi-7} we get that
\begin{align}\label{poi-8}
&\int_{\mathcal{D}_n}
\exp\left\{
-\frac{n^{1/N}\log n}{N}S({\mathbf z})
+n^{1/N}Q({\mathbf z})
\right\}
\,{\rm d}{\mathbf z}
\nonumber\\
&\quad=
\exp\left\{
n^{1/N}Q({\mathbf z}_0)
+
O\!\left(\frac{n^{1/N}}{\log n}\right)
\right\}
\int_{\mathcal{D}_n}
\exp\!\left\{
-\frac{n^{1/N}\log n}{N}S({\mathbf z})
\right\}
\,{\rm d}{\mathbf z}.
\end{align}

Now, let us estimate the effect of the moving domain ${\mathcal{D}_n}$.
Since \( S({\mathbf z}) \geqslant z_1+\ldots+z_{N-1}, \) we have that
\begin{equation}\label{eq:S-integrable}
\int_{\mathcal{D}} {\rm e}^{-S({\mathbf z})} \,{\rm d}\mathbf{z}\leqslant 1,
\end{equation}
and, hence,  by~\eqref{eq:S-gap}, we get
 \begin{align*} & \int_{\mathcal{D}\setminus \mathcal{D}_n} \exp\left\{ -\frac{n^{1/N}\log n}{N} S({\mathbf z}) \right\} \,{\rm d}\mathbf{z}\\
 &\quad=\int_{\mathcal{D}\setminus \mathcal{D}_n} \exp\left\{ -\left(\frac{n^{1/N}\log n}{N}-1\right) S({\mathbf z}) \right\}\exp\big\{-S(\mathbf{z})\big\} \,{\rm d}\mathbf{z}\\ &\quad\leqslant \exp\left\{ -\left( \frac{n^{1/N}\log n}{N}-1 \right)(N+\eta) \right\} 
 \end{align*} for sufficiently large $n$.
 \\
 On the other hand, by~integrating over a sufficiently small ball around \({\mathbf z}_0\) and using the local quadratic upper bound \eqref{eq:S-local} for \(S({\mathbf z})-N\), similarly as in the derivation of \eqref{poi-9}, we obtain
 \[ \int_\mathcal{D} \exp\left\{ -\frac{n^{1/N}\log n}{N} S({\mathbf z}) \right\} \,{\rm d}\mathbf{z} \geqslant \exp\left\{ -n^{1/N}\log n -c_3\log n \right\} 
 \]
with some positive quantity $c_3$ not depending on $n$.
 Consequently, by~the last two derived estimates, we get that  
 \[ \int_{\mathcal{D}_n} \exp\left\{ -\frac{n^{1/N}\log n}{N} S({\mathbf z}) \right\} \,{\rm d}\mathbf{z} \mathop{\sim}\limits_{n\rightarrow\infty} \int_\mathcal{D} \exp\left\{ -\frac{n^{1/N}\log n}{N} S({\mathbf z}) \right\} \,{\rm d}\mathbf{z}. 
\]
Therefore, the~preceding estimate  \eqref{poi-8} implies that 
\begin{align*}
&\int_{{\mathcal{D}_n}}
\exp\left\{
-\frac{n^{1/N}\log n}{N}S({\mathbf z})
+
n^{1/N}Q({\mathbf z})
\right\}
\,{\rm d}\mathbf{z}
\\
&\quad\mathop{\sim}\limits_{n\rightarrow\infty}
\exp\left\{
n^{1/N}Q({\mathbf z}_0)
+
O\!\left(\frac{n^{1/N}}{\log n}\right)
\right\}
\int_\mathcal{D}
\exp\!\left\{
-\frac{n^{1/N}\log n}{N}S({\mathbf z})
\right\}
\,{\rm d}\mathbf{z}.
\end{align*}
Since
\(
Q({\mathbf z}_0)
=
N+\log\prod_{i=1}^N\lambda_i,
\)
according to the expression \eqref{aa-1}, we have that
\begin{align*}
\widetilde{I}_n
&\mathop{\sim}\limits_{n\rightarrow\infty}
n^{(N-1)/N}
\exp\left\{
-\sum_{i=1}^N\lambda_i
+
n^{1/N}
\left(
N+\log\prod_{i=1}^N\lambda_i
\right)
+
O\!\left(\frac{n^{1/N}}{\log n}\right)
\right\}
\\
&\quad\times
\int_\mathcal{D}
\exp\!\left\{
-\frac{n^{1/N}\log n}{N}
S({\mathbf z})
\right\}
\,{\rm d}\mathbf{z}.
\end{align*}

Conditions {\rm (i)--(iii)} of Lemma~\ref{lem:wong-multidimensional}
follow from relations  \eqref{eq:S-integrable}, \eqref{eq:S-gap}, and~\eqref{eq:S-hessian}, respectively. Hence we can apply the assertion of Lemma~\ref{lem:wong-multidimensional} with
\[
d=N-1,\quad
f=S,\quad
g\equiv1,\quad
\mathbf{x}_0={\mathbf z}_0,
\quad\text{and}\ \  
\lambda=\frac{n^{1/N}\log n}{N}.
\]
Since
\(
S({\mathbf z}_0)=N
\), and~(see \eqref{eq:S-hessian})
\[
\det\left(
\frac{\partial^2 S(\mathbf{z}_0)}{\partial z_i\partial z_j}
\right)_{i,j=1}^{N-1}
=N,
\]
 by Lemma~\ref{lem:wong-multidimensional} we get 
\[
\begin{aligned}
&\int_\mathcal{D}
\exp\left\{
-\frac{n^{1/N}\log n}{N}
S({\mathbf z})
\right\}
\,{\rm d}{\mathbf z}
\\
&\quad\mathop{\sim}\limits_{n\rightarrow\infty}
(2\pi)^{(N-1)/2}
N^{(N-2)/2}
n^{-(N-1)/(2N)}
(\log n)^{-(N-1)/2}
\exp\!\left\{-n^{1/N}\log n\right\}.
\end{aligned}
\]
This relation and  the preceding expression for \(\widetilde{I}_n\) imply that
\[
\begin{aligned}
\widetilde{I}_n
&\mathop{\sim}\limits_{n\rightarrow\infty}
(2\pi)^{(N-1)/2}
N^{(N-2)/2}
(\log n)^{-(N-1)/2}
n^{(N-1)/(2N)}
\\
&\quad\times
\exp\left\{
-n^{1/N}\log n
+n^{1/N}
\left(
N+\log\prod_{i=1}^N\lambda_i
\right)
-\sum_{i=1}^N\lambda_i
+
O\!\left(\frac{n^{1/N}}{\log n}\right)
\right\}.
\end{aligned}
\]
Recalling that
\[
\mathcal P_{C,n}
=
\widetilde{I}_n\exp\!\bigl\{O(\log n)\bigr\},
\]
(see \eqref{par2}), and~
\[
\mathcal P_n^{(N)}\mathop{\sim}\limits_{n\rightarrow\infty}\mathcal P_{C,n},
\]
(see \eqref{opa2}),
the above asymptotic formula for \(\widetilde{I}_n\) yields the desired asymptotic relation
\eqref{eq:pnm-main}. This finishes the proof of the main theorem.
}

\section{Illustrative~Example}

As an illustration of Theorem~\ref{t1}, we consider the case \(N=3\) with
\[
\xi_1\sim \mathrm{Poisson}(1.39),\quad
\xi_2\sim \mathrm{Poisson}(1.14),\quad
\xi_3\sim \mathrm{Poisson}(1.58),
\]
where \(\xi_1,\xi_2,\xi_3\) are independent. Figure~\ref{fig:poisson-product-log} compares the Monte Carlo values of \(\log \mathcal{P}_n^{(3)}\) with the explicit logarithmic asymptotic formula, and~Figure~\ref{fig:poisson-product-ratio} shows the corresponding~ratio.

We compare logarithmic probabilities because the
remainder in our asymptotic formula is not $o(1)$,
so the relative error in approximating the probabilities
themselves may be~large.

{
Figure~\ref{fig:poisson-product-log} shows that both
the Monte Carlo estimates and the asymptotic expression
decrease as $n$ increases. However, a~noticeable gap
remains between them over the considered range of $n$.
This illustrates that the asymptotic formula captures
the general behavior of the logarithmic tail probability,
but is not particularly accurate for these values of $n$.}
{
This is expected, since the remainder term in our
logarithmic asymptotic formula is
$O(n^{1/N}/\log n)$ rather than $o(1)$.
}
\begin{figure}[H]

\includegraphics[width=0.82\textwidth]{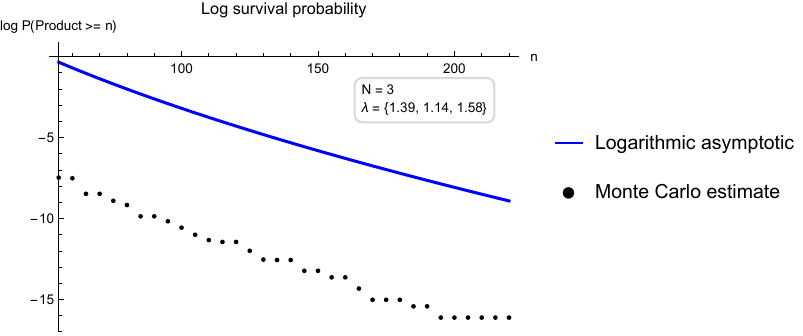}
\caption{Comparison of Monte Carlo values of \(\log \mathcal{P}_n^{(3)}\) with the explicit part of the logarithmic asymptotic~formula.}
\label{fig:poisson-product-log}
\end{figure}
\unskip

\begin{figure}[H]

\includegraphics[width=0.82\textwidth]{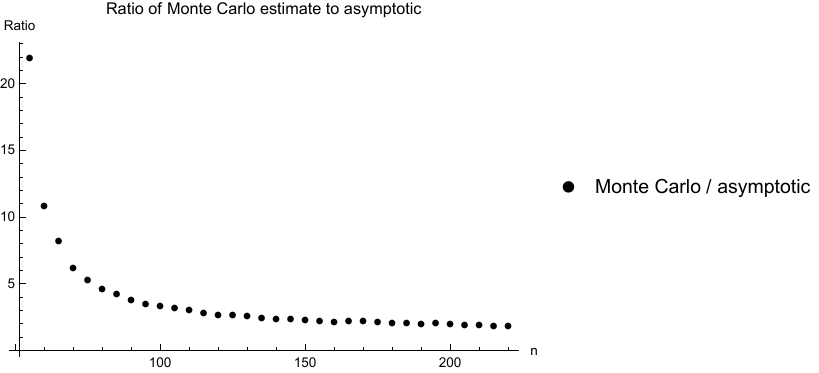}
\caption{Ratio of the Monte Carlo estimate of \(\log \mathcal{P}_n^{(3)}\) to the explicit logarithmic asymptotic~formula.}
\label{fig:poisson-product-ratio}
\end{figure}

{
Figure~\ref{fig:poisson-product-ratio} shows that the
ratio of the Monte Carlo estimate to the logarithmic
asymptotic expression decreases as $n$ increases.
However, the~ratio remains a little bit  above 1 over
the considered range. This suggests that larger values
of $n$ may be needed to observe the asymptotic
behavior predicted by Theorem~\ref{t1}. However, for~larger values of $n$, the~tail probabilities
become extremely small, making reliable estimation
by standard Monte Carlo simulation computationally
impractical, thus leaving us with the analytical
asymptotic~formula.
}

\section{Concluding~Remarks}

{The main result of this paper is established in Theorem \ref{t1}, particularly in formula \eqref{eq:log-pnm}, which provides an asymptotic representation of the logarithm of the tail probability
\[
\log {\mathcal{P}_n^{(N)}}
=\log\mathbb{P}\left(\prod_{k=1}^N \xi_k\geqslant n\right),
\]
where $\xi_1,\ldots,\xi_N$ are independent Poisson random variables with parameters $\lambda_1,\ldots,\lambda_N$, respectively. In~particular, the~obtained asymptotic result shows that the product of independent Poisson random variables exhibits a Weibull-type discrete heavy-tailed behavior. More precisely,
$$
\log{\mathcal{P}_n^{(N)}}
\mathop{\sim}\limits_{n\rightarrow\infty} -n^{1/N}\log n.
$$
This result shows how the product's tail behavior changes with the number of~factors.

An important observation is that the leading asymptotic term does not depend on the collection of the Poisson parameters $\lambda_1,\ldots,\lambda_N$. These parameters affect only the lower-order terms in the asymptotic expansion. Thus, the~principal tail behavior is determined by the number of independent Poisson random variables being multiplied rather than by their individual~parameters.

The case $N=1$ provides a useful reference point for the general result. Indeed, applying Stirling's formula yields
\begin{align}\label{bbb}
\log\mathbb{P}(\xi_1\geqslant n)
=-n\log n+n(1+\log\lambda_1)
-\frac12\log(2\pi n)-\lambda_1+o(1).
\end{align}
Consequently, the~leading term agrees with the general asymptotic relation obtained in Theorem \ref{t1}. Moreover, formula \eqref{bbb} provides a more precise asymptotic description in the single-factor~case.

The results obtained in this paper also suggest several directions for further research. In~particular, it would be desirable to derive a more precise asymptotic formula for $\mathcal{P}_n^{(N)}$
for arbitrary fixed $N$. More specifically, the~aim would be to improve the remainder term in formula \eqref{eq:pnm-main} of Theorem \ref{t1} from $O(n^{1/N}/\log n)$ to $o(1)$. Such a refinement would provide an asymptotic representation analogous to \eqref{bbb}, but~valid for an arbitrary number of factors. The~present analysis indicates that achieving this goal may require techniques that take greater account of the discrete structure of the underlying random variables, rather than relying exclusively on continuous asymptotic~approximations.

More broadly, the~study of products of independent Poisson random variables illustrates how multiplication can substantially transform tail behavior and provides a framework for developing refined multidimensional asymptotic methods for discrete probability distributions.}

\vspace{+6pt}

\authorcontributions{Conceptualization, J.\v{S}.; methodology,  D.C. and J.\v{S}.; software, D.C. and R.P.; validation,  R.P. and J.\v{S}.; formal analysis, R.P.; investigation, D.C.; resources, J.\v{S}.; data curation, R.P.; writing---original draft preparation, D.C.; writing---review and editing, R.P. and J.\v{S}.; visualization, D.C.; supervision, J.\v{S}.; project administration, J.\v{S}.; funding acquisition, D.C. and R.P. All authors have read and agreed to the published version of the~manuscript.}

\funding{This research received no external~funding.}

\dataavailability{The original contributions presented in this study are included in the paper. Further inquiries can be directed to the corresponding author.} 


\conflictsofinterest{The authors declare no conflicts of~interest.}

\begin{adjustwidth}{-\extralength}{0cm}

\reftitle{References}

\PublishersNote{}
\end{adjustwidth}

\end{document}